\documentclass[11pt]{amsart}
\pdfoutput=1
\usepackage[leqno]{amsmath}
\usepackage{amsmath,amsfonts,amsthm,amssymb}
\usepackage{epsfig}
\usepackage{graphicx}
\usepackage{cite}
\usepackage{epsfig}
\usepackage{color}
\usepackage{amsmath}
\usepackage{amssymb}
\newtheorem{lemma}{Lemma}
\newtheorem{proposition}{Proposition}
\newtheorem{theorem}{Theorem}
\newtheorem{corollary}{Corollary}

\theoremstyle{definition}

\theoremstyle{remark}

\newcommand{\Link}{\mathop{\mathrm{Link}}}
\newcommand{\Inc}{\mathop{\mathrm{Inc}}}
\newcommand{\CAT}{\mathop{\mathrm{CAT}}}

\begin{document}

\thispagestyle{empty}

\title{Ramified rectilinear polygons:\\ coordinatization by dendrons}

\author{Hans--J\"urgen Bandelt}
\address{Dept. of Mathematics, University of Hamburg, Bundesstr. 55, D-20146 Hamburg, Germany}
\email{bandelt@math.uni-hamburg.de}

\author{Victor Chepoi}
\address{Laboratoire d'Informatique Fondamentale, Universit\'e d' Aix-Marseille, Facult\'e des Sciences de Luminy, F-13288 Marseille Cedex 9, France}
\email{chepoi@lif.univ-mrs.fr}

\author{David Eppstein}
\address{Computer Science Department, University of California, Irvine, Irvine  CA 92697-3435, USA}
\email{eppstein@ics.uci.edu}

\begin{abstract}
Simple rectilinear polygons  (i.e. rectilinear polygons without holes or cutpoints)
can be regarded as finite rectangular cell complexes coordinatized by two
finite dendrons. The intrinsic $l_1$-metric
is thus inherited from the product of the two finite dendrons via
an isometric embedding. The rectangular cell complexes that share this same
embedding property are called ramified rectilinear polygons. The
links of vertices in these cell complexes may be arbitrary
bipartite graphs, in contrast to simple rectilinear polygons where
the links of points are either 4-cycles or paths of length at most
3. Ramified rectilinear polygons are particular instances of
rectangular complexes obtained from cube-free median graphs, or equivalently
simply connected rectangular complexes with triangle-free links. The underlying
graphs of finite ramified rectilinear polygons can be recognized among graphs
in linear time by a Lexicographic Breadth-First-Search.
Whereas the symmetry of a simple rectilinear polygon is very
restricted (with automorphism group being a subgroup of the dihedral
group $D_4$), ramified rectilinear polygons are universal: every finite
group is the automorphism group of some ramified rectilinear
polygon.
\end{abstract}
\maketitle

\section{Introduction}

Polygons endowed with a metric are fundamental objects in
computational  and distance geometry. Computation of the geodesic distance
between two query points of a polygon $P$ (that is, the length of the shortest
path within $P$ that connects the two points) is required in  a variety
of algorithmic problems motivated by applications in robot motion,
plant and facility layout, urban transportation, and wire layout; for a
survey, see \cite{Mi}. To answer such queries efficiently,  the
polygon $P$ may be subdivided into simple pieces such as
triangles, rectangles or trapezoids at the preprocessing stage. If $P$
is a simple polygon (bounded by a polygonal Jordan curve), then the
dual graph of this subdivision will generally be a tree $T(P)$. Given a pair of
points $s$ and $t$, the cells of the subdivision containing $s$ and $t$
may be computed using point-location methods~\cite{BCKO}. On the other hand, the
dual graph of the subdivision gives a rough idea of the global
location of $s$ and $t$ in the polygon $P$, yielding essential
information for computing the exact geodesic
distance between $s$ and $t$. In particular, if $P$ is simple, and the boundaries between cells of the subdivision are straight line segments, then
any shortest $(s,t)$-path traverses the cells of the subdivision along the unique path of $T(P)$
that connects the cells containing $s$ and $t$. Therefore the
subdivision of $P$ can be viewed as a kind of coordinatization of
the points of $P$.

A simple rectilinear polygon (also called an orthogonal polygon) is a simple polygon $P$ for which all
sides (boundary segments) are parallel to the coordinate axes; we view such polygons as being topologically closed (that is, they include their boundary segments) and endowed with
the $l_1$-metric.
For simple rectilinear polygons the idea of coordinatization by trees can
be made more precise, as we describe below, so that geodesic distances can be directly computed
from the tree-coordinates of the points. However, this more precise coordinatization involves two trees
rather than one, using a subdivision of the polygon into rectangles that does not have a single tree as its dual.

\begin{figure}[t]
\scalebox{0.8}{\includegraphics{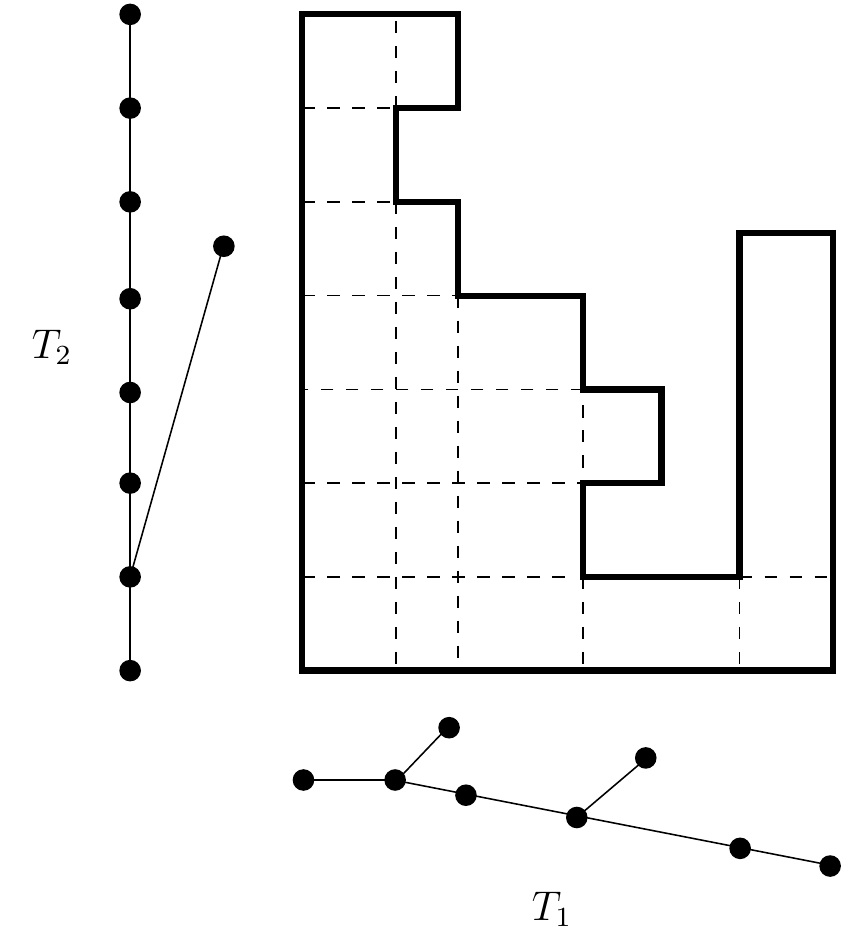}}
\caption{A simple rectilinear  polygon $P$ and
its grid lines giving rise to the coordinate tree
networks $T_1$ and $T_2.$}
\label{polygon}
\end{figure}

A \emph{corner} of $P$ is a point
on the boundary of $P$ incident with two (perpendicular) sides. We
call a horizontal or vertical axis-parallel line segment a \emph{line} of $P$ if it lies entirely in $P$ and is maximal with respect
to inclusion. A \emph{grid line} of $P$ is any line of $P$ passing
through some corner of $P$. The \emph{grid network} $N(P)$ has all
points in which horizontal and vertical grid lines intersect as its
vertices (thus including all corners); in computational geometry terminology, it is the arrangement formed by the grid lines. Two vertices $u$ and $v$ of
$N(P)$ are adjacent exactly when they are consecutive on a common
line, so that $N(P)$ forms a planar graph; the edge $uv$ of $N(P)$ is then weighted by the length of the
line segment between $u$ and $v$.
The shortest-path metric $d_N$ of the grid network $N(P)$ can be
extended to a metric $d_P$ of the entire simple polygon $P$ in the
following way. For any two distinct points $s$ and $t$ of $P$
augment the set of grid lines by the lines of $P$ that pass through $s$ or
$t$. Define the expanded grid network $N_{s,t}(P)$ with respect to
this enlarged set of distinguished lines, so that $d_P(s,t)$ is the
distance in the network $N_{s,t}(P)$. A shortest path in
$N_{s,t}(P)$ joining $s$ and $t$ can be turned into a rectilinear
path $\pi_{st}$ of $P$ having length $d_P(s,t)$
by replacing each edge of $N_{s,t}(P)$ by the corresponding axis-parallel line segment in $P$.  In
general, a \emph{rectilinear path} $\pi$ of $P$ is any polygonal
chain consisting of axis-parallel segments lying inside $P$, and its
length is the sum of lengths of the constituent segments of $P$. It
is easy to see that $d_P(s,t)$ equals the \emph{intrinsic $l_1$-distance} (or geodesic distance) of $P$ between $s$ and $t$, that is, the minimum length of
a rectilinear path connecting $s$ and $t$.

Partitioning the set of grid lines into the horizontal and the
vertical lines gives rise to two tree networks, $T_1$ and $T_2;$ see
Fig.~\ref{polygon} for an example. Namely, the vertices of $T_1$ are the
vertical grid lines, where two lines are adjacent exactly when they
support opposite edges of some 4-cycle of $N(P).$ The weight of the
respective edge of $T_1$ is then the horizontal distance between the
two vertical lines, which is the length of the perpendicular edges
of the supported 4-cycle of $N(P).$ Similarly, $T_2$ is defined in
terms of the horizontal grid lines. Every vertex $x$ of $N(P)$ thus
receives two coordinates $x_1$ and $x_2$ in $T_1$ and $T_2,$
respectively, such that the mapping $x\mapsto (x_1,x_2)$ constitutes
an isometric embedding of the network $N(P)$ into the Cartesian
product $T_1\,\Box\,T_2$ of the tree networks  $T_1$ and $T_2.$

If every edge of $T_1$ and $T_2$, respectively, is replaced by a segment
of the real line having the same length, one obtains  finite
\emph{dendrons} $D_1$ and $D_2$ (for definitions, see Section~\ref{sec:frameworks}). The points of $D_1$ are then in
one-to-one correspondence with the vertical lines of $P,$ and the
points of $D_2,$ mutatis mutandis, with the horizontal lines. The
isometric embedding of $N(P)$ into the Cartesian product $T_1\,\Box\,
T_2$ therefore lifts to an isometric embedding of $P$ into the
product space $D_1\times D_2$. After linear time
preprocessing, distance queries in trees and dendrons may be performed in
constant time using lowest common ancestor algorithms
\cite{ahu_lca,ht_lca,agkr_lca}. Since the $l_1$ geodesic distance between any two points in $P$ is just the sum of distances between the corresponding points in the two factors of the Cartesian product, this shows that distance computations in simple
rectilinear polygons are reduced to finding the coordinates of the
query points in the factors.

This coordinatization of simple rectilinear polygons motivates us to  consider
and investigate the rectilinear complexes endowed with the intrinsic
$l_1$-metric that have isometric embeddings into Cartesian
products of two dendrons. We will call such objects ``ramified
rectilinear polygons.'' A straightforward inductive argument shows that every tree with more than one node can occur
in the coordinatization of a simple rectilinear polygon.
However, not every pair of trees may occur in combination; for instance, if $T_1$ is the star $K_{1,k}$, then $T_2$ must have at least $\max(k+1,2k-2)$ nodes, whereas this constraint does not apply in the ramified case. Therefore, ramified rectilinear polygons are a strict generalization of simple rectilinear polygons.

In this paper we characterize the
ramified rectilinear polygons and their underlying networks (the networks playing the same role for these complexes as the network $N(P)$ played for the simple rectilinear polygon $P$); see Section~\ref{sec:results} for details of this characterization. The underlying networks of ramified rectilinear polygons are exactly the networks that can be  isometrically embedded
into the Cartesian product of two tree networks; our
characterization provides a simple linear time algorithm for
their recognition. Note, however, that for every $k\ge 3$ the problem of
recognizing the graphs isometrically embeddable into the Cartesian
product of $k$ trees is NP-complete \cite{BaVdV}. For a survey of
other relationships between the metric geometry of complexes and
that of graphs, see \cite{BaCh_survey}.  More generally, in a subsequent paper,
we investigate the complete
Menger-convex  metric spaces that can be isometrically embedded into
the products of two dendrons. We call such spaces ``partial double
dendrons."

\section{The discrete and geometric frameworks}
\label{sec:frameworks}

\subsection{Metric spaces} Let $(X,d)$ be a metric space.
An \emph{arc} joining two points $x,y$ of $X$ is a homeomorphic  map
$\gamma: [0,l]\rightarrow X$ such that $\gamma(0)=x$ and
$\gamma(l)=y.$  Then $(X,d)$ is called \emph{arc-connected} if any
two points of $X$ can be joined by an arc. Given $x\in X$ and $r>0,$
the open ball $\{ y\in X: d(x,y)<r\}$ of radius $r$ centered at $x$
is denoted by $B^{\circ}(x,r),$ and the corresponding closed ball
$\{ y\in X: d(x,y)\le r\}$ is denoted by $B(x,r).$  The \emph{interval} between two points $x,y$ of $X$ is the set $I(x,y)=\{ z\in
X: d(x,y)=d(x,z)+d(z,y)\}$; for example, in Euclidean spaces, the interval $I(x,y)$ is the closed line segment having $x$ and $y$ as its endpoints. The space $(X,d)$ is called \emph{Menger-convex} if for  any two distinct points $x,y\in X$ there
exists a point $z\in I(x,y)$ different from $x,y$. If, in addition to being Menger-convex, $(X,d)$ is
\hbox{(Cauchy-)complete}, it follows that for any two points $x,y$ and every $0\le
t\le 1$ there exists a point $z\in X$ such that $d(x,z)=t\cdot
d(x,y)$ and $d(z,y)=(1-t)\cdot d(x,y)$ \cite{Blu}.  A \emph{geodesic}
joining two points $x$ and $y$ from $X$ is the image of a
(continuous) map $\gamma$ from a line segment $[0,l]\subset \mathbb{R}$
to $X$ such that $\gamma(0)=x, \gamma(l)=y$ and
$d(\gamma(t),\gamma(t'))=|t-t'|$ for all $t,t'\in [0,l].$ The space
$(X,d)$ is said to be \emph{geodesic} if every pair of points $x,y\in
X$  is joined by a geodesic (which is necessarily included in
$I(x,y)$) \cite{BrHa,Pa}. Since every geodesic is an arc, a geodesic
space is arc-connected. Every complete Menger-convex metric space is
geodesic \cite{Blu,Me}.

A \emph{dendron} $D$ is a complete Menger-convex metric space in
which any two distinct points $x,y$ can be joined by a unique arc
$\gamma(x,y).$ Note that $\gamma(x,y)$ is the unique geodesic
joining the points $x,y.$ In fact, $\gamma(x,y)$ coincides with
$I(x,y).$ To see this, assume for a contradiction that there exists a point $z$ that is in $I(x,y)$
but outside $\gamma(x,y)$. Let $\gamma(x,z)$ and $\gamma(y,z)$ be the
two geodesics connecting the points $x$ and $y$ with $z.$  The
choice of $z$ in $I(x,y)$ implies that $\gamma(x,z)\cap
\gamma(y,z)=\{ z\}.$ Therefore concatenating $\gamma(x,z)$ and
$\gamma(z,y)$ will lead to an arc joining $x$ and $y$ different from
$\gamma(x,y)$, contradicting the assumption that arcs are unique.  This establishes that indeed $\gamma(x,y)=I(x,y)$. A
\emph{leaf} of a dendron $D$ is a point whose removal does not
separate any pair of points and a \emph{ramification point} is a
point whose removal creates at least three arc-connected components.
The leaves and ramifications points are called the \emph{vertices} of
$D.$ A \emph{finite dendron} \cite{Bo} is a compact dendron with a
finite number of vertices. A finite dendron $D$ may be formed from a finite tree $T$ by replacing each edge of $T$ by a line segment; the vertices of $D$ correspond to vertices of $T$, but $T$ may have additional degree-two vertices that do not correspond to vertices in $D$.

The \emph{(direct or Cartesian) product} of two metric spaces $(X_1,d_1)$ and
$(X_2,d_2)$ is the set $X_1\times X_2$ endowed with the metric
$d((x_1,x_2),(y_1,y_2))=d_1(x_1,y_1)+d_2(x_2,y_2).$ A function $f:
X\rightarrow X'$ between two metric spaces $(X,d)$ and $(X',d')$ is
an  \emph{isometric embedding} of $X$ into $X'$ when
$d'(f(x),f(y))=d(x,y)$ for any $x,y\in X.$ In this case $Y:=f(X)$ is
called an \emph{(isometric) subspace} of $X'.$ We also refer to the product
$D=D_1\times D_2$ of two dendrons $D_1,D_2$ as a
\emph{double dendron,} for short. A \emph{partial double dendron} is
then a complete Menger-convex isometric subspace of a double
dendron.

A subspace $Y$ of a metric space $(X,d)$ is \emph{gated} if for every
point $x\in X$ there exists a (unique) point $x'\in Y,$ the \emph{gate} of $x$ in $Y$,  such that $d(x,y)=d(x,x')+d(x',y)$ for all
$y\in Y$ \cite{DrSch}. The intersection of two gated subspaces is a
gated subspace again. $(X,d)$ is said to be a {\it gated amalgam} of two smaller spaces $(Y,d_Y)$ and $(Z,d_Z)$ along a common
nonempty gated subspace $Y\cap Z$ if $X=Y\cup Z$ and $d$ extends $d_Y$ and $d_Z$ with
$$d(y,z):=d_Y(y,y')+d_Z(y',z)=d_Y(y,z')+d_Z(z',z)   \text{ for all } y\in Y \text{ and } z\in Z,$$
where  $x\mapsto x'$  denotes the gate map from $Y$ as well as from $Z$ to $Y\cap Z.$ A particular instance of this kind
of amalgamation is given when $Z$ is the product of some subspace $U$ of $Y$ with the real interval $[0,\lambda]$ of length  $\lambda>0$
where the fiber  $U\times\{ 0\}$ gets identified with $U.$ The gated amalgam of $Y$ and the product $Z=U\times [0,\lambda]$ along $U=Y\cap Z$
is then referred to as the {\it gated expansion} of $Y$ by $[0,\lambda]$ along $U.$ Gated subspaces are necessarily interval-convex, where a subspace $Y$ of $X$ is called \emph{interval-convex} if $I(x,y)\subseteq Y$ for any $x,y\in Y$. The \emph{convex hull} conv$(Z)$ of $Z\subset X$ is the smallest convex
subspace containing $Z.$ A \emph{half-space} $H$ of $X$ is a convex
subspace with a convex complement. The partition $\{ H,X\setminus
H\}$ is called a \emph{convex split} of $X$.

Let $(X,d)$ be a metric space and $x,y,z\in X.$ We often use
``median" sets of the type $m(x,y,z)=I(x,y)\cap I(y,z)\cap I(z,x).$
If $m(x,y,z)$ is a singleton for all $x,y,z\in X$, then the space
$X$ is called \emph{median} \cite{VdV} and we usually refer to
$m(x,y,z)$ as to the \emph{median} of $x,y,z$ (here we do not
distinguish between the singleton and the corresponding point).
Dendrons are median spaces. Indeed, given three points $x,y,z$ of a
dendron $D$, the union  $\gamma(x,y)\cup \gamma(x,z)$ contains an
arc between $y$ and $z$ which is composed of two subarcs of
$\gamma(x,y)$ and $\gamma(x,z)$ intersecting in a single point $m.$
By uniqueness of arcs, the composed arc coincides with the geodesic
$\gamma(y,z)$ \cite{Me}. Therefore, as geodesics are intervals, $m$
is the unique median of $x,y,z.$  The product of any two median  spaces is
evidently median, and therefore products of dendrons are median. But in
general, an isometric subspace $Y$ of a median space $X$ need not be
median, unless  it is \emph{median-stable}, that is, $m(x,y,z)\in Y$
for all $x,y,z\in Y.$


A \emph{geodesic triangle} $\Delta (x_1,x_2,x_3)$ in a geodesic
metric space $(X,d)$ consists of three distinct points in $X$ (the
vertices of $\Delta$) and a geodesic  between each pair of vertices
(the sides of $\Delta$). A \emph{comparison triangle} for $\Delta
(x_1,x_2,x_3)$ is a triangle $\Delta (x'_1,x'_2,x'_3)$ in the
Euclidean plane ${\mathbb E}^2$ such that $d_{{\mathbb
E}^2}(x'_i,x'_j)=d(x_i,x_j)$ for $i,j\in \{ 1,2,3\}.$ A geodesic
metric space $(X,d)$ is defined to be a \emph{$\CAT(0)$ space} \cite{Gr}
if all geodesic triangles $\Delta (x_1,x_2,x_3)$ of $X$ satisfy the
comparison axiom of Cartan--Alexandrov--Toponogov:

\medskip\noindent
\emph{If $y$ is a point on the side of $\Delta(x_1,x_2,x_3)$ with
vertices $x_1$ and $x_2$ and $y'$ is the unique point on the line
segment $[x'_1,x'_2]$ of the comparison triangle
$\Delta(x'_1,x'_2,x'_3)$ such that $d_{{\mathbb E}^2}(x'_i,y')=
d(x_i,y)$ for $i=1,2,$ then $d(x_3,y)\le d_{{\mathbb
E}^2}(x'_3,y').$}

\medskip\noindent
This simple axiom turns out to be very powerful, because $\CAT(0)$
spaces can be characterized  in several natural ways (for
a full account of this theory consult the book \cite{BrHa}). $\CAT(0)$
spaces  play a vital role in modern combinatorial group theory,
where various versions of hyperbolicity are related to
group-theoretic properties \cite{GhHa,Gr}; many arguments in this
area have a strong metric graph-theoretic flavor. A geodesic metric
space $(X,d)$ is $\CAT(0)$ if and only if any two points of this space
can be joined by a unique geodesic. $\CAT(0)$ is also equivalent to
convexity of the function $f:[0,1]\rightarrow X$ given by
$f(t)=d(\alpha (t),\beta (t)),$ for any geodesics $\alpha$ and
$\beta$ (which is further equivalent to convexity of the
neighborhoods of convex sets). This implies that $\CAT(0)$ spaces are
contractible.

\subsection{Graphs and networks}
\label{sec:graphs-and-networks}
Connected undirected graphs $G=(V,E)$ endowed with the graph
metric $d_G$ defined by
$$d_G(x,y)=\mbox{ the length of a shortest path between } x
\mbox{ and } y$$ are basic examples of discrete metric spaces. An
\emph{induced subgraph} $G'=(V',E')$ of $G$ is a subset of vertices
with the induced edge-relation, and it is an \emph{isometric
subgraph} if the distances  between any two vertices of $V'$ in $G$
and $G'$ coincide. These notions naturally carry over to \emph{networks,} i.e. graphs with weighted edges.

A graph $G$ is a \emph{median graph} if $(V,d_G)$ is a median space~\cite{Av_ternary}.
Discrete median spaces, in general, can be regarded as median
networks: a \emph{median network} is a median graph with weighted
edges such that opposite edges in any 4-cycle have the same length \cite{Ba_condorcet}.
Typical examples of median graphs are hypercubes (weak Cartesian powers of $K_2$), trees, and
squaregraphs (i.e., plane graphs in which all inner faces are
4-cycles and all inner vertices have degrees $\ge 4$ \cite{BaChEpp_square,ChDrVa}).
Finite median graphs are exactly the graphs which are obtained from
finite hypercubes by applying successive gated amalgamations
\cite{Is,VdV_matching}. It is well known \cite{BaVdV,Mu1} that median
graphs isometrically embed into hypercubes and therefore into
Cartesian products of trees. A median graph contains an isometric
6-cycle if and only if it contains an induced cube $Q_3$, for the two remaining cube vertices are forced to exist as medians of alternating subsets of vertices from the 6-cycle. Median
graphs or networks not containing any induced cube (or cube network,
respectively) are called \emph{cube-free.}

The isometric embedding of a median graph $G$ into a (smallest)
hypercube coincides with the so-called canonical embedding, which is
determined by the Djokovi\'c-Winkler relation $\Theta$ on the edge
set of $G:$ two edges $uv$ and $wx$ are $\Theta$-related exactly
when
$$d_G(u,w)+d_G(v,x)\ne d_G(u,x)+d_G(v,w);$$ see \cite{EppFaOv,ImKl}. For a
median graph this relation is transitive and hence an equivalence
relation. It is the transitive closure of the ``opposite" relation
of edges on 4-cycles: in fact, any two $\Theta$-related edges can be
connected by a ladder (viz., the Cartesian product of a path with
$K_2$), and the block of all edges $\Theta$-related to some edge
$uv$ constitute a cutset $\Theta(uv)$ of the median graph, which
determines one factor of the canonical hypercube \cite{Mu1,Mu}. The
cutset $\Theta(uv)$ defines a convex split $\sigma(uv)=\{
W(u,v),W(v,u)\}$ of $G$ \cite{Mu,VdV}, where $W(u,v)=\{ x\in X:
d(u,x)<d(v,x)\}$ and $W(v,u)=V-W(u,v).$ Conversely, for every convex
split of a median graph $G$ there exists at least one edge $xy$ such
that $\{ W(x,y),W(y,x)\}$ is the given split. Two convex splits
$\sigma_1=\{ A_1,B_1\}$ and $\sigma_2=\{A_2,B_2\}$ of $G$ are said
to be \emph{incompatible} if all four intersections $A_1\cap A_2,
A_1\cap B_2, B_1\cap A_2,$ and $B_1\cap B_2$ are non-empty, and are
called \emph{compatible} otherwise. The \emph{incompatibility graph}
$\Inc(G)$ of $G$ has the convex splits of $G$ as vertices and  pairs
of incompatible convex splits as edges.

The convex hull conv$(X)$ of any finite set $X$ of vertices in an
infinite median graph $G$ is known to form a finite median subgraph of $G$
\cite{VdV}. Therefore $G$ is a directed union of its finite convex
subgraphs. This will be used in many arguments for transferring a
result from finite to infinite median graphs: in particular, any
property that can be expressed in terms of finitely many convex sets
or splits carries over from the finite to the infinite case.

\begin{figure}[t]
\scalebox{0.40}{\includegraphics{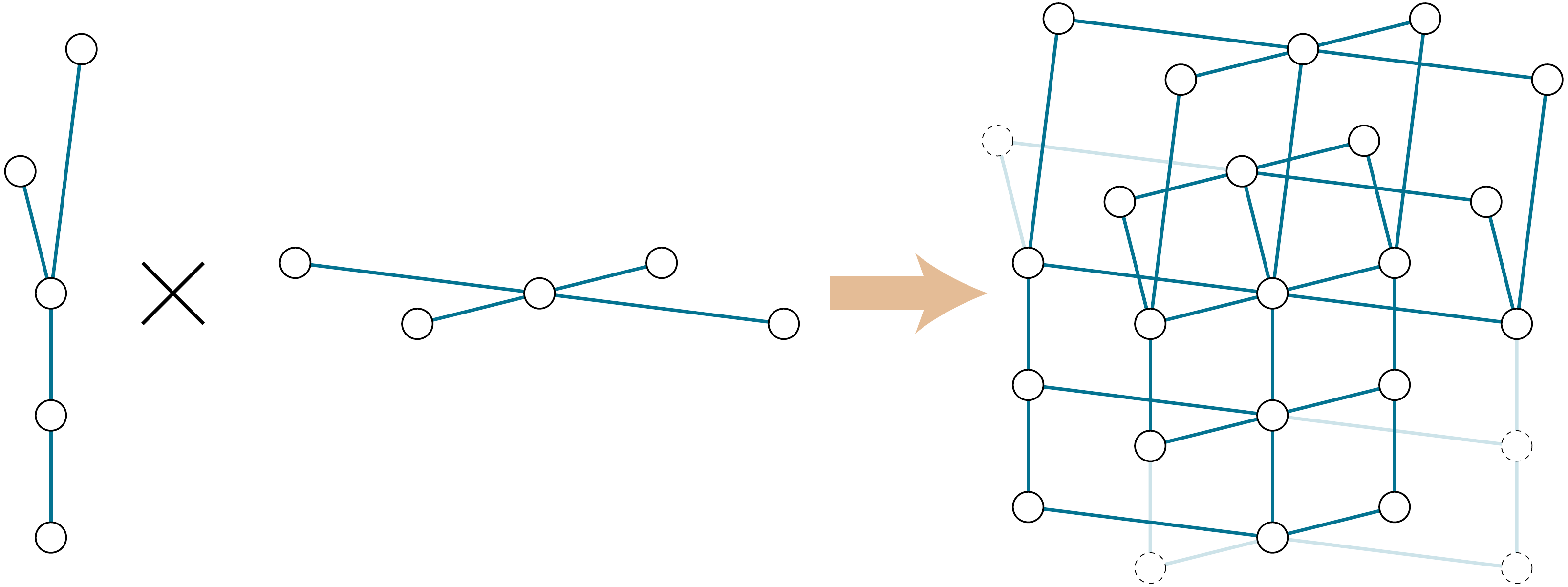}}
\caption{An isometric subgraph of the Cartesian
product of two trees}
\label{example}
\end{figure}

\subsection{Rectangular complexes and ramified rectilinear polygons}
In this subsection we recall and specify some terminology about cell
complexes; for general notions, see \cite{BrHa,VdV}.    A \emph{rectangular complex} $\mathcal K$ is a
2-dimensional cell complex $\mathcal K$ whose 2-cells are isometric to axis-parallel rectangles of
the $l_1$-plane; unlike some related work~\cite{Epp_embed} we require that the 1-cells of the complex be complete edges of these rectangles, so the nonempty intersection of any two rectangles must either be a vertex or an edge of both of them. If all 1-cells of
$\mathcal K$ have equal length, then we call $\mathcal K$ a
\emph{square complex}; in this case we may assume without loss of generality that
the squares of the complex are all unit squares. Square complexes are the 2-dimensional
instances of \emph{cubical complexes,} viz. the cell complexes (where cells have
finite dimension)  in
which every cell of dimension $k$ is isometric to the unit cube of
${\mathbb R}^k.$ It is not stipulated here that the cell complexes
have only finitely many cells, but, as in \cite{BrHa}, we require
that the complex contains only finitely many isometry types of rectangular
cells.

The 0-dimensional faces of a rectangular complex $\mathcal K$ are called its
{\it vertices}, forming the vertex set $V({\mathcal K})$ of $\mathcal K$. The {\it underlying network}
$N({\mathcal K})$ of $\mathcal K$ then represents the 1-dimensional faces by weighted {\it edges} of the same length.
Disregarding edge lengths (by adopting unit lengths) one obtains the underlying graph $G({\mathcal K}) = (V({\mathcal K}),E({\mathcal K}))$.
In the case of a simple rectilinear polygon $P,$ the cells of the associated rectangular complex ${\mathcal K}(P)$
are determined by a pair of adjacent vertical grid lines together with a pair of adjacent horizontal
grid lines. Then the underlying network of ${\mathcal K}(P)$ is the grid network $N(P)$ of $P.$ Conversely, from any graph $G$
or network $N$ one can derive a cube or box complex $|G|$ or $|N|$ by replacing all cuboids, i.e. subgraphs
(or subnetworks) of $G$ (or $N$) isomorphic to (edge-weighted) cubes of any dimensions
(vertices, edges, induced 4-cycles, etc.) by solid boxes. The complexes $|G|$ and $|N|$ are referred to as the geometric
realization of $G$ and $N,$ respectively.
In particular, ${\mathcal K}(P)=|N(P)|$ for every simple rectilinear polygon $P.$  More generally,
median networks give rise to particularly interesting cubical complexes of higher dimensions.


A cell complex $\mathcal K$ is called \emph{simply connected} if it
is connected and every continuous mapping of the 1-dimensional
sphere ${\mathbb S}^1$ into $\mathcal K$ can be extended to a
continuous mapping of the disk ${\mathbb D}^2$ with boundary
${\mathbb S}^1$ into $\mathcal K$. The \emph{link} of a vertex $x$ in
$\mathcal K$ is the graph $\Link(x)$ whose vertices are the 1-cells
containing $x$ and where two 1-cells are adjacent if and only if
they are  contained in  a common 2-cell (see \cite{BrHa} for the
notion of link in general polyhedral complexes). The \emph{link
graph} $\Link(\mathcal K)$ of $\mathcal K$ is then the union of the
graphs $\Link(x)$ for all vertices $x$ of $\mathcal K.$

A rectangular complex ${\mathcal K}$  can be endowed with several
intrinsic metrics \cite{BrHa} transforming ${\mathcal K}$ into a
complete geodesic space. Suppose that inside every 2-cell of
${\mathcal K}$ the distance is measured according to an $l_1$- , $l_2$-, or $l_{\infty}$-
metric (or, more generally, according to any
$l_p$-metric). The \emph{intrinsic} $l_1$- or $l_2$-\emph{metric}
of $\mathcal K$ is defined by assuming
that the distance between two points $x,y\in {\mathcal K}$ equals the infimum
of the lengths of the paths joining them. Here a \emph{path} in $\mathcal K$
from $x$ to $y$ is a sequence $P$ of points
$x=x_0,x_1\ldots x_{m-1}, x_m=y$ such that for each $i=0,\ldots,
m-1$ there exists a 2-cell $R_i$ containing $x_i$ and $x_{i+1};$ the
\emph{length} of $P$ is $l(P)=\sum_{i=0}^{m-1} d(x_i,x_{i+1}),$ where
$d(x_i,x_{i+1})$ is computed inside $R_i$ according to the
respective metric. A \emph{ramified rectilinear polygon} is a finite
rectangular complex $\mathcal K$  endowed with the intrinsic
$l_1$-metric which embeds isometrically into the product of two
finite dendrons.

\begin{figure}[t]
\includegraphics[scale=0.75]{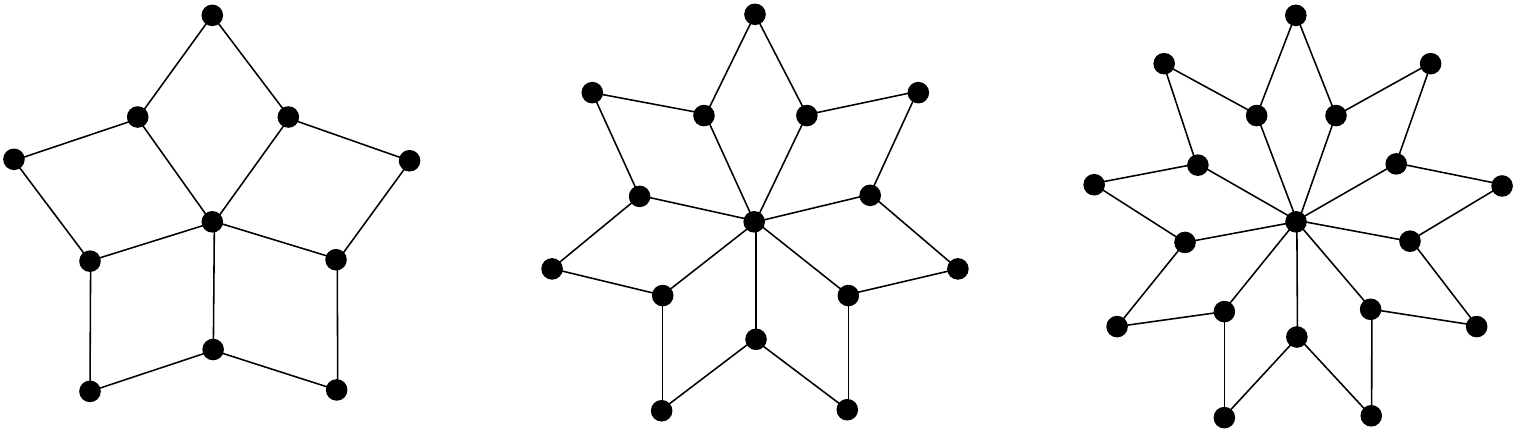}
\caption{The first three odd cogwheels}
\label{odd-wheels}
\end{figure}

\section{Main results}
\label{sec:results}
In this section we formulate  the main results of the paper; in
particular, we provide geometric and combinatorial characterizations
of ramified rectilinear polygons, along with their underlying graphs
and, more generally, arbitrary partial double dendrons.
We start with a characterization  of  graphs $G$ which are
isometrically embeddable in the Cartesian product of two trees (see
Fig.~\ref{example} for an example). We call such graphs \emph{partial double
trees.} Such graphs are bipartite and cube-free, that is,  the cube
$Q_3$ does not occur as an induced subgraph.  For a graph $F$ denote
the \emph{simplex graph} \cite{BaVdV} of $F$ by $\kappa(F)$:
it has
the simplices (or cliques) of $F$ as its vertices, and two vertices in $\kappa(F)$ are connected
by an edge when their simplices differ by the presence or absence of exactly one vertex of $F$.  The cube $Q_3$
is isomorphic to $\kappa(C_3),$ and when $F=C_n$ is a cycle with $n>3$ vertices, then
$\kappa(F)$ is called a \emph{cogwheel} \cite{KlKo} or bipartite wheel
\cite{Ba_hmg}.  A median graph is cube-free exactly when it
does not contain a subgraph isomorphic to $Q_3-v,$ the cube minus a vertex.
There exist, however, cube-free median graphs that are not partial
double trees: each odd cogwheel $\kappa(C_{2k+1})$ for $k\ge 2$ (Fig.~\ref{odd-wheels}) is
a cube-free isometric subgraph of $Q_{2k+1},$ yet the least number
$\tau(G)$ of tree factors in a Cartesian representation of the
graph $G=\kappa(C_{2k+1})$ equals 3. The graphs $\kappa(C_{2k+1})$ play a key role
in our characterization, as they signify odd cycles in the links of
vertices in the complex $|G|.$

\begin{theorem} \label{partial-double-tree} For a connected graph $G$ the
following conditions are equivalent:

\begin{itemize}
\item[(i)] $G$ is a partial double tree;
\item[(ii)] $G$ is a median subgraph of the Cartesian product of two trees;
\item[(iii)] $G$ is a median graph such that $\Inc(G)$ is bipartite;
\item[(iv)] $G$ is a median graph such that $\Link(|G|)$ is bipartite;
\item[(v)] $G$ is a median graph such that $\Link(x)$ is bipartite
for all vertices $x$ of $G;$
\item[(vi)] $G$ is a median graph that does not contain $Q_3-v$ nor any
$\kappa(C_{2k+1})$ $(k\ge 2)$ as an isometric subgraph;
\item[(vii)] $G$ is an isometric subgraph of a hypercube
such that $C_{2k+2}$ and $\kappa(C_{2k+1})$ $(k\ge 2)$ are not isometric
subgraphs of $G;$
\item[(viii)] $C_3,K_{2,3},C_{k+3},$ and $\kappa(C_{2k+1})$ $(k\ge 2)$
are not isometric subgraphs of $G$.
\end{itemize}
\end{theorem}

In this theorem, the isometric embedding of a partial double dendron into the Cartesian product of two trees
is in general not unique, even when the two projections are onto (that is, the two trees are minimal).
Uniqueness of the embedding relative to permutation and automorphisms of the two trees is achieved under
2-connectivity:

\begin{corollary} \label{2-connected-partial-double-tree} For a connected graph $G$ with at least three vertices,
the following conditions are equivalent:

\begin{itemize}
\item[(i)] $G$ is a partial double tree such that $G$ has a unique isometric embedding into the Cartesian
product of two trees relative to minimality, permutations and automorphisms of the two trees;
\item[(ii)] $G$ is 2-connected median subgraph of the Cartesian product of two trees;
\item[(iii)] $G$ is a median graph such that $\Inc(G)$ is bipartite and connected;
\item[(iv)] $G$ is a median graph such that $\Link(x)$ is bipartite and connected for all vertices $x$ of $G;$
\item[(v)]  $G$ is a median graph such that $\Link(G)$ is bipartite and connected.
\end{itemize}
\end{corollary}

The characterization of partial double trees provided by Theorem \ref{partial-double-tree}(v) can be used to
recognize  partial double trees in
polynomial time. As we will show below, this task can be implemented in
linear time.

\begin{theorem} \label{recognition} For a finite graph $G=(V,E),$ one can decide in linear time $O(|V|+|E|)$ whether
$G$ is a partial double tree.
\end{theorem}

The ease of recognizing partial double trees, as codified by Proposition~\ref{recognition}, stands in contrast with the difficulty of recognizing the generalization of partial double trees to cube-free median graphs. If $G$ is a triangle-free graph, then the graph formed from a graph $G$ by adding a new vertex adjacent to every vertex in $G$, and by subdividing every edge in $G$, is the simplex graph of $G$ and is a cube-free median graph. But if $G$ is not triangle-free, then the graph derived from it in this way is not a median graph. Hence, recognition of cube-free median graphs is no easier than (and is in fact equivalent in computational difficulty to) recognition of triangle-free graphs, for which no linear time algorithms are known~\cite{ImrKlaMul}.

We continue with a characterization of the rectangular complexes
that arise from cube-free median graphs. These complexes are more
general than ramified rectilinear polygons; nevertheless they share
one significant property with the latter: equipped with the
$l_1$-metric they define metric spaces that are median.

\begin{theorem} \label{cube-free} For a rectangular complex $\mathcal K$ the following conditions  are
equivalent:

\begin{itemize}
\item[(i)] the underlying graph $G({\mathcal K})$ of ${\mathcal K}$ is a cube-free median graph;
\item[(ii)] $\mathcal K$ equipped
with the intrinsic $l_1$-metric $d$ is median;
\item[(iii)] $\mathcal K$ equipped with the intrinsic $l_2$-metric is $\CAT(0)$;
\item[(iv)] $\mathcal K$ is simply connected and for every vertex
$x\in V({\mathcal K}),$ the graph $\Link(x)$ is $C_3$-free.
\end{itemize}
If any of these conditions holds for the rectangular complex ${\mathcal K},$ then the metric space ${\mathcal K}$ coincides with the geometric realization $|N({\mathcal K})|$ of its network.
\end{theorem}


The next result characterizes the ramified rectilinear polygons and
the partial double dendrons. As a consequence, one obtains a
characterization of all metric spaces isometrically embeddable into
double dendrons.

\begin{theorem} \label{ramified-polygon} For a finite rectangular complex $\mathcal K$ equipped
with the intrinsic $l_1$-metric $d$ the following conditions  are
equivalent:

\begin{itemize}
\item[(i)] $\mathcal K$ is a ramified rectilinear polygon;
\item[(ii)] $\mathcal K$ is median and each vertex $x\in V({\mathcal K})$ has some closed
neighborhood $B(x,\epsilon)$ (where $\epsilon>0$) that is a partial
double dendron;
\item[(iii)] $\mathcal K$ is simply connected and each vertex $x\in V({\mathcal K})$ has some closed
neighborhood $B(x,\epsilon)$ (where $\epsilon>0$) that is a partial
double dendron;
\item[(iv)] $\mathcal K$ is simply connected and for each vertex $x\in V({\mathcal K}),$
the graph $\Link(x)$ is bipartite;
\item[(v)] the underlying graph $G({\mathcal K})$ of ${\mathcal K}$ is a partial double tree and ${\mathcal K}=|G({\mathcal K})|$;
\item[(vi)] the metric space $({\mathcal K},d)$ can be obtained from a singleton space by a
finite sequence of gated expansions by real intervals along gated dendrons supporting exclusively even cogfans;
\item[(vii)] the metric space $({\mathcal K},d)$ can be obtained from finitely many rectangles
(equipped with the $l_1$-metric) via successive gated amalgamations along gated dendrons that
in each point support cogfans in either part of the same parity exclusively.
\end{itemize}

\end{theorem}

We will finally show that median graphs are universal with respect to
their automorphism groups. This is not surprising in view of
Birkhoff's Theorem stating that every group is the automorphism
group of some distributive lattice (and the close relationship
between distributive lattices and median algebras~\cite{BirKis-BotAMS-47}). What might,
however, be surprising at first sight is the fact that the partial
double trees are already universal.

\begin{theorem} \label{automorphism}  Every group is the automorphism
group aut$(G)$ of a 2-connected partial double tree $G$ of radius 2. If the
group is finite, the graph can be taken to be finite. The
automorphism group of $G$ is isomorphic to the group of isometries of
the geometric realization $|G|$, which is a ramified rectilinear
polygon without cutpoints in the case that $G$ is finite.
\end{theorem}


\section{Proof of Theorem \ref{partial-double-tree} and Corollary \ref{2-connected-partial-double-tree}}
\label{sec:ptpdt}

We begin with a known characterization of the cube-free median graphs (see \cite{Ba_hmg}): a connected graph is a cube-free median
graph if and only if it does not contain $C_3$, $K_{2,3}$, or $C_{k+3}$
(for $k\ge 2$) as isometric subgraphs.
Because the relation of being an isometric subgraph is transitive, this implies that an isometric subgraph of a cube-free median graph
is itself median. Because trees are median graphs and the product operation preserves median graphs, the Cartesian product $H$ of any
finite family of trees is a median graph; if the number of factors equals two, then $H$ is cube-free. Thus, the partial double trees
are isometric subgraphs of cube-free median graphs, and are themselves median, giving us the equivalence (i)$\Longleftrightarrow$(ii)
of Theorem~\ref{partial-double-tree}.

The equivalence (ii)$\Longleftrightarrow$(iii) of Theorem \ref{partial-double-tree} is a particular
case of a result of \cite{BaVdV} establishing  that a median graph
$G$ embeds into the Cartesian product of $n$ trees if and only if
$\Inc(G)$ is $n$-colorable. For each vertex $x$ of $G,$  $\Link(x)$ is an induced subgraph of $\Link(|G|),$ therefore
if $\Link(|G|)$ is bipartite, then all links of vertices are bipartite as well, inferring that (iv)$\Longrightarrow$(v).
Since a graph is bipartite exactly when it does not contain an isometric odd cycle, the condition (vi) is a reformulation of (v), whence
the equivalence (v)$\Longleftrightarrow$(vi) holds as well.

Note that every isometric
subgraph of a hypercube is bipartite and does not contain any
induced $K_{2,3}.$ By combining this fact with the characterization of cube-free median graphs by forbidden isometric subgraphs and the characterization of cube-free median graphs as the median graphs without $Q_3-v$, we see immediately that the equivalences
(vi)$\Longleftrightarrow$(vii)$\Longleftrightarrow$(viii) indeed
hold. It remains to show that the conditions (ii), (iv),
and (vi) are equivalent, as we do in the rest of this section.
At the end of the section, an example will indicate that, in contrast to (v), there does not exist a characterization of isometric subgraphs of products of three trees in terms of the links of their vertices.

Since the canonical embedding \cite{GrWi,ImKl} of a median graph $G$ into a hypercube is    governed by its transitive relation $\Theta$, which is the transitive hull of the relation   formed by opposite edges from 4-cycles, every isometric embedding of $G$ into a Cartesian product of graphs can be     expressed   in terms of (improper) edge-coloring such that opposite edges in every 4-cycle are   equally colored. With the factors of the corresponding Cartesian product being canonical images of $G$, this constitutes a subdirect representation of the associated   median algebras. For the sake of minimizing the necessary background information,
the next two lemmas provide a direct proof.

\begin{lemma}
A median graph $G$ is a subgraph of a Cartesian product of $n$ graphs $H_i$ if and only if the edges of $G$ may be (improperly) colored with $n$ colors in such a way that every two opposite edges of a 4-cycle of $G$ are assigned the same color. If the projection $\pi_i$ from $G$ to each factor $H_i$ is onto, then each factor $H_i$ must itself be a median graph, and a graph isomorphic to $H_i$ may be recovered by contracting every edge of $G$ that does not have color~$i$.
\end{lemma}

\begin{proof}
First we observe that subgraphs of products have the edge colorings described in the lemma. To see this, consider
any subgraph $G$ of the Cartesian product $H$ of
graphs $H_i$ $(i=1,\ldots,n)$, and let $\pi_i$ denote the projection of $G$ onto each of the factors;
we may assume without loss of generality that each projection $\pi_i$ is onto, for otherwise we may replace $H_i$ by $\pi_i(G)$.
Then for each pair $u,v$ of adjacent vertices of $G$ there is
a unique index $i$ such that $\pi_i$ does not collapse the edge
$uv.$ This can be visualized within $G$ by coloring the edge $uv$
with  color $i.$ If this is done for all edges, then we get an
$n$-coloring of the edges of $G$.
The coloring may be improper: adjacent edges may have the same color.
If a 4-cycle $C$ in $G_i$ is colored with a single color, then clearly it satisfies the conditions of the lemma. Otherwise, consider the set $E_C$ of edges in subgraphs $H_i$ onto which the edges of the cycle project. This set of edges must have cardinality two, for $C$ is a subgraph of the product of the edges in $E_C$, a hypercube of dimension at most four, but the only 4-cycles in a hypercube are products of two of its factors. In this case, the edges of $C$ have alternating colors, again meeting the conditions of the lemma.

Conversely, suppose that the edges of a median graph $G$ are colored in such a way that every 4-cycle is monochromatic or alternatingly colored. Let
$E_1,\ldots,E_n$ be the partition of the edges of $G$ defined by
this coloring. In a median graph $G$, as discussed in Section~\ref{sec:graphs-and-networks}, the cutsets of $G$ consisting of the edges between two
complementary half-spaces are the connected components of the ``opposite'' relation on 4-cycles, and each cutset is therefore monochromatic. In other words, this
$n$-coloring is essentially a cutset coloring in the sense of
\cite{BaVdV,BaVdV_superext}, where we do not distinguish between colors
corresponding to different edges of the same factor.
From this coloring, define $\pi_i$ to be a function from $G$ to a graph $H_i$ that contracts all edges not belonging to $E_i$. Every median graph $G$ is a subgraph of a hypercube, formed by taking the product of a number of copies of $K_2$ equal to the number of cutsets of $G$, and $\pi_i$ respects this hypercube structure, so $H_i$ is a subgraph of a hypercube with dimension $|E_i|$ and $G$ is a subgraph of the product of the graphs $H_i$. The median $m(u,v,w)$ of any three vertices in $H_i$ may be obtained by finding representatives of $u$, $v$, and $w$ in $G$ and using $\pi_i$ to project the median onto $H_i$; therefore, each $H_i$ is a median-closed subgraph of a hypercube and therefore is itself a median graph.
\end{proof}

Alternatively, for each color $i$ we may define a ``congruence'' $\psi_i$ on the vertex set
of $G:$ vertices $x$ and $y$ are congruent modulo $\psi_i$ if and
only if they are joined by a shortest path whose edges are not
colored with color $i.$ Then $H_i$ is obtained from $G$ by
identifying all vertices $x$ and $y$ that are congruent modulo
$\psi_i,$ with two different components being adjacent when
connected by an edge from $E_i.$
Motivated by the algebraic theory of median algebras we say that
$G$ has a subdirect representation in terms of the graphs $H_i\cong
G/\psi_i,$ in symbols:
$$G\hookrightarrow \Pi_{i=1}^n G/\psi_i.$$

\begin{lemma}
If a median graph $G$ is (improperly) edge-colored so that all of its cutsets are monochromatic,
then $G$ has a monochromatic cycle if and only if it has a monochromatic 4-cycle.
\end{lemma}

\begin{proof}
Obviously, if $G$ has a monochromatic 4-cycle then it has a monochromatic cycle. In the other direction,
suppose that $C$ is a monochromatic cycle in $G.$
Then there is a cutset $\Theta_1$ of $G$ containing at least two
edges $w_0w_1$ and $w_2w_3$ of $C,$ where the indices are chosen so
that $w_0$ and $w_3$ are joined by a shortest path $P$ avoiding this
cutset. Now any cutset $\Theta_2$ of $G$ containing some edge  from
$P$ must also contain an edge from $C,$ thus the edges of this
cutset have the same color as the edges of $C.$ From the choice of
the cutset $\Theta_2$ we conclude that the convex splits
$\sigma_1=\{ A_1,B_1\}$ and $\sigma_2=\{ A_2,B_2\}$ defined by
$\Theta_1$ and $\Theta_2$, respectively, are incompatible. Notice
that in a median graph $G$ two such convex splits are incompatible
if and only if \emph{there exists a 4-cycle $(x_1,x_2,x_3,x_4,x_1)$
of $G$ such that $x_1\in A_1\cap A_2, x_2\in B_1\cap A_2, x_3\in
B_1\cap B_2,$ and $x_4\in A_1\cap B_2.$} Indeed, since the sets
$A_1\cap A_2$ and $B_1\cap A_2$ are convex and their union is the
convex set $A_2,$ we can find an edge $xy\in \Theta_1$ such that
$x\in A_1\cap A_2$ and $y\in B_1\cap A_2.$ Analogously, we can find
an edge $x'y'\in \Theta_1$ such that $x'\in A_1\cap B_2$ and $y'\in
B_1\cap B_2.$ Since the edges $xy$ and $x'y'$ can be connected by a
ladder and $x,y\in A_2, x',y'\in B_2,$ necessarily this ladder
contain a 4-cycle $(x_1,x_2,x_3,x_4,x_1)$ such that $x_1,x_4\in A_1,
x_2,x_3\in B_1$ and $x_1,x_2\in A_2, x_3,x_4\in B_2,$ establishing
our assertion. Since $x_1x_2,x_3x_4\in \Theta_1$ and
$x_1x_4,x_2x_3\in \Theta_2$ and all edges of the cutsets $\Theta_1$
and $\Theta_2$ have the same color, we conclude that the 4-cycle
$(x_1,x_2,x_3,x_4,x_1)$ is monochromatic.
\end{proof}

Now, if all graphs $H_i=G/\psi_i$ are trees, then there
are no monochromatic 4-cycles in $G.$ Conversely, if every 4-cycle of
$G$ is non-monochromatic  and all cutsets of $G$ are monochromatic,
then each $H_i$ must be a tree. To show this, first notice that any
pair of convex splits defined by two cutsets of  color $i$ are
compatible, or else from the previous assertion we will obtain a
monochromatic 4-cycle. Let $A_1$ be a minimum by inclusion
half-space participating in a convex split $\{ A_1,B_1\}$ whose
cutset $\Theta_1$ has color $i.$ Then necessarily all edges of the
subgraph induced by $A_1$ are not colored in color $i,$ whence $A_1$
is a pendant vertex of the graph $H_i.$ Employing the induction
hypothesis   to the coloring of the subgraph of $G$ induced by the
convex set $B_1,$ we conclude that $H_i$ is indeed a tree.

Finally,  notice that an $n$-coloring of edges of a median graph $G$
such that  the opposite edges of each 4-cycle  have the same color
and the incident edges have different colors is equivalent to the
$n$-coloring of the link graph $\Link(|G|).$ In particular, we infer
that $G$ is a median subgraph of a Cartesian product of two trees if
and only if $\Link(|G|)$ is bipartite.

We summarize the preceding observations in the following result.

\begin{proposition} \label{product-n-trees} A median graph $G$ is a median subgraph of a
Cartesian product of trees $T_1,\ldots,T_n$ if and only if $G$
admits an $n$-coloring of its edges such that  every $C_4$ in $G$ is
dichromatic with opposite edges having the same color.  The required
trees $T_i$ are then obtained from $G$ by collapsing all edges
colored with a color different from $i$. In particular, a median
graph $G$ is a median subgraph of a Cartesian product of two trees
if and only if the link graph $\Link(|G|)$  is bipartite.
\end{proposition}

This result establishes the equivalence
(ii)$\Longleftrightarrow$(iv) and the implication
(ii)$\Longrightarrow$(vi) of Theorem \ref{partial-double-tree}. Indeed, it is easy to see
why $\kappa(C_{2k+1})$ ($k\ge 2$) cannot be embedded in the Cartesian
product of only two trees. Since monochromatic 4-cycles are
forbidden, consecutive spokes of $\kappa(C_{2k+1})$ (i.e. edges incident with the
vertex of degree $2k+1$) have to differ in color, and if only two colors were allowed, they would have to alternate in color. But, as $2k+1$
is odd, two colors do not suffice. Therefore if a median graph $G$
contains $Q_3$ or $\kappa(C_{2k+1})$ $(k\ge 2)$ as an isometric subgraph,
then $\tau (G)\ge 3.$

To complete the proof of Theorem~\ref{partial-double-tree}, it remains to establish that (vi)$\Longrightarrow$(iii). Let $G$ be
a median graph. Recall that for every convex split of $G$ there exists at least
one edge $xy$ such that $\{ W(x,y),W(y,x)\}$ is the given split.
Recall also that the half-spaces of the simplex graph $\kappa(F)$ of a graph $F$
are of the form $H_v=\{ K\in V(\kappa(F)): v\in K\}$ and $V(\kappa(F))-H_v,$ for each
$v\in V(F)$ \cite{BaVdV}. Hence two convex splits $\{
H_v,V(\kappa(F))-H_v\}$ and $\{ H_w,V(\kappa(F))-H_w\}$ are incompatible if and
only if $v$ and $w$ are adjacent, that is, $\Inc(\kappa(F))\cong F.$ As we
noticed above, two incompatible convex splits of a median graph are
associated with at least one 4-cycle $(u,v,w,x,u)$ such that $\{
W(u,x),W(x,u)\}$ and $\{ W(x,w),W(w,x)\}$ are given splits (where
$ux,xw$ are edges of the cycle). Then it is clear that the median
graph $G'$ obtained from a median graph $G$ by contracting the edges
between two complementary half-spaces $W(y,z)$ and $W(z,y)$ has
$\Inc(G)-\{ W(y,z),W(z,y)\}$ as its incompatibility graph (up to
isomorphism).

\begin{lemma} \label{simplex_graph} Let $G$ be a finite median graph. Assume that for
every triple $s_0,s_1,s_2\in V(Inc(G))$ of compatible (convex)
splits there exists an induced path from $s_1$ to $s_2$ in $\Inc(G)$
which does not contain any neighbor of $s_0$ (i.e., any convex split
incompatible with $s_0$). Then $G$ is a simplex graph such that
$G\cong Inc(\kappa(G)).$
\end{lemma}

\begin{proof} Suppose there exists a chain $H_1\subset H_0\subset H_2$ of
three distinct half-spaces in $G.$ Let $H'_i$ $(i=1,\ldots, k)$ be a
sequence of half-spaces such that $H'_1=H_1, H'_k=H_2,$ and $\{
H'_i,V(G)-H_i\}$ is incompatible with  $\{ H'_{i+1},V(G)-H'_{i+1}\}$
and compatible (but distinct) with  $\{ H_0,V(G)-H_0\} (i=1,\ldots,
k-1).$ Then, necessarily, either $H'_2$ or $V(G)-H'_2$ is properly
contained in $H_0$ because the corresponding split is compatible
with $\{ H_0,V(G)-H_0\}$ but incompatible with $\{ H_1,V(G)-H_1\}.$
Continuing this way along the given induced path, we eventually
conclude that either $H'_{k-1}$ or $V(G)-H'_{k-1}$ is properly
included in $H_0.$ But then this set is also contained in
$H_2=H'_k,$ contrary to the incompatibility of the two splits
corresponding to $H'_{k-1}$ and $H_2.$ Therefore all chains of
half-spaces of $G$ have at most two members. This means that the
depth (sensu  \cite{BaVdV_superext}) of $G$ is at most 2, that is,  $G$ is a
simplex-graph, i.e., $G\cong \kappa(F)$ and $\Inc(G)\cong$ $\Inc(\kappa(F))\cong F,$
as required. This completes the proof of the lemma.
\end{proof}

Now we are in position to conclude the proof of the implication
(vi)$\Longrightarrow$(iii), or equivalently, its contrapositive: a median graph that does not satisfy (iii)
also does not satisfy (vi). Thus, assume that $G$ is a median graph the incompatibility graph of which contains an odd cycle (failing to satisfy (iii)) and let
$C_{2k+1}$ ($k\ge 1$) be a shortest odd cycle in the incompatibility
graph $\Inc(G)$ of the convex splits of $G.$ We assert that either $G$ is not cube-free, or $G$
contains a convex subgraph isomorphic to $\kappa(C_{2k+1})$, in either case failing (vi). Since the convex hull of the
4-cycles in which the incompatible splits of $C_{2k+1}$ cross
is finite, we can assume without loss of generality that the
median graph $G$ itself is finite.

Contract successively the edges between complementary half-spaces
not belonging to the given $C_{2k+1}.$ This median graph is then
isomorphic to $\kappa(C_{2k+1})$ by virtue of the Lemma
\ref{simplex_graph}. Now we reverse the procedure by adding all
removed convex splits back one by one, thus performing \emph{Mulder's
expansion procedure} \cite{Mu1,Mu}. In the case that $k=1$ the successive contractions produce a cube
$Q_3\cong \kappa(C_3),$ and in each expansion we retain a copy of this
cube. Trivially, a cube subgraph is always convex in a median graph.
In the remaining cases, $k\ge 2$; suppose that we have a convex subgraph $\kappa(C_{2k+1})$
at a certain step, as we will be guaranteed to have after all of the contraction steps. Expanding the corresponding larger graph (where
$\kappa(C_{2k+1})$ lives in)  induces a convex expansion of $\kappa(C_{2k+1}),$
possibly a trivial one. We cannot expand along a convex set which
contains some $C_4$ of $\kappa(C_{2k+1}),$ for otherwise, we would get a
cube. Therefore the only nontrivial case is attained when we expand
only two non-subsequent spokes of $\kappa(C_{2k+1}).$ The center $x$ of
$\kappa(C_{2k+1})$ then gets copied: $x'$ and $x''$ are then centers of two
simplex graphs of cycles. The number of spokes incident with $x'$
plus the number of spokes incident with $x''$ equals $(2k+1)+4$ and
hence is again an odd number. So one of the simplex graphs is of the
form $\kappa(C_l)$ with $l$ odd number and $l\le 2k+1.$ By minimality of
$k,$ we therefore obtain a convex $\kappa(C_{2k+1})$ again in the expanded
graph. Thus, $G$ itself must contain either a cube (for $k=1$) or a convex $\kappa(C_{2k+1})$ (in the remaining cases),
and it fails to satisfy property (vi). This completes the implication from the failure of (iii) to the
failure of (vi) and completes the proof of Theorem \ref{partial-double-tree}.

To derive Corollary \ref{2-connected-partial-double-tree} from Theorem \ref{partial-double-tree}, we establish an auxiliary result
which characterizes 2-connectedness for arbitrary median graphs. This immediately yields Corollary \ref{2-connected-partial-double-tree}, observing that
any cutpoint $x$ of $G$ would allow to switch the roles of the two trees in exactly one of the composite subgraphs amalgamated at $x$ in thus obtain an
essentially different embedding.

\begin{lemma} \label{Inc_connected} For a median graph $G$ with at least three vertices, the following conditions are equivalent:
\begin{itemize}
\item[(i)]  $G$ is 2-connected;
\item[(ii)] $\Inc(G)$ is connected;
\item[(iii)] $\Link(x)$ is connected for all vertices $x$ of $G;$
\item[(iv)] $\Link(G)$ is connected.
\end{itemize}
\end{lemma}

\begin{proof}  If $G$ is not 2-connected, it is the amalgam of two nontrivial median graphs $G_1$ and $G_2$ along some cutpoint $x.$ Then $\Inc(G),$ $\Link(x),$ and $\Link(G)$
are each the disjoint union of the corresponding graphs of $G_1$ and $G_2$ and hence disconnected.

If $G$ is 2-connected, then for any incident edges $xy$ and $xz$ of $G$ there exists a path joining $y$ and $z$ that avoids $x.$ Choose such a path $P$ with smallest distance sum to $x.$
If the maximum distance to $x$ of vertices on $P$ is larger than 2, then choose any such vertex $v$  and replace it by the median $w$ of $x$ and the two neighbors of $v$ on $P.$ Since $v$ and $w$ are at distance 2, one thus obtains a new path still avoiding $x$ but with smaller distance sum to $x$ than $P$, contrary to the choice of $P$. Therefore all vertices on $P$ have distances 1 or 2 to $x$, so that $P$ together with $x$ induces a cogfan in $G.$ This cogfan testifies to a path between $xy$ and $xz$ in $\Link(x)$ and to one between the corresponding convex splits in $\Inc(G).$ Considering two non-incident edges $wy$ and $xz$, paths connecting them in $\Link(G)$ and $\Inc(G)$ are provided by virtue of a straightforward induction on the minimal distance, say $d(w,x),$ between the endpoints of $wy$ and $xz$ in $G.$
\end{proof}

 In Fig.~\ref{critical} we present an isometric subgraph $H_n$ of the Cartesian product of four trees
that is critical in the sense that taking any proper isometric subgraph would   result in a graph
which needs only three trees for an embedding into a Cartesian  product of trees. Moreover,
collapsing any block of $\Theta$ yields either $H_{n-1}$ (as long as $n>0$) or a graph embeddable
into a Cartesian product of three trees. The median graph $G$   depicted in Fig.~\ref{critical-median}
is the median graph generated by the graph $H_1$ of Fig.~\ref{critical} and therefore
cannot be embedded into the product of three trees, although every  subgraph of $G$ that is a simplex graph needs
no more than three trees for an embedding.
Thus, this example shows that, in contrast to Theorem~\ref{partial-double-tree}, 3-colorability of local neighborhoods does not suffice to characterize isometric subgraphs of Cartesian products of three trees.

\begin{figure}[t]
\scalebox{0.50}{\includegraphics{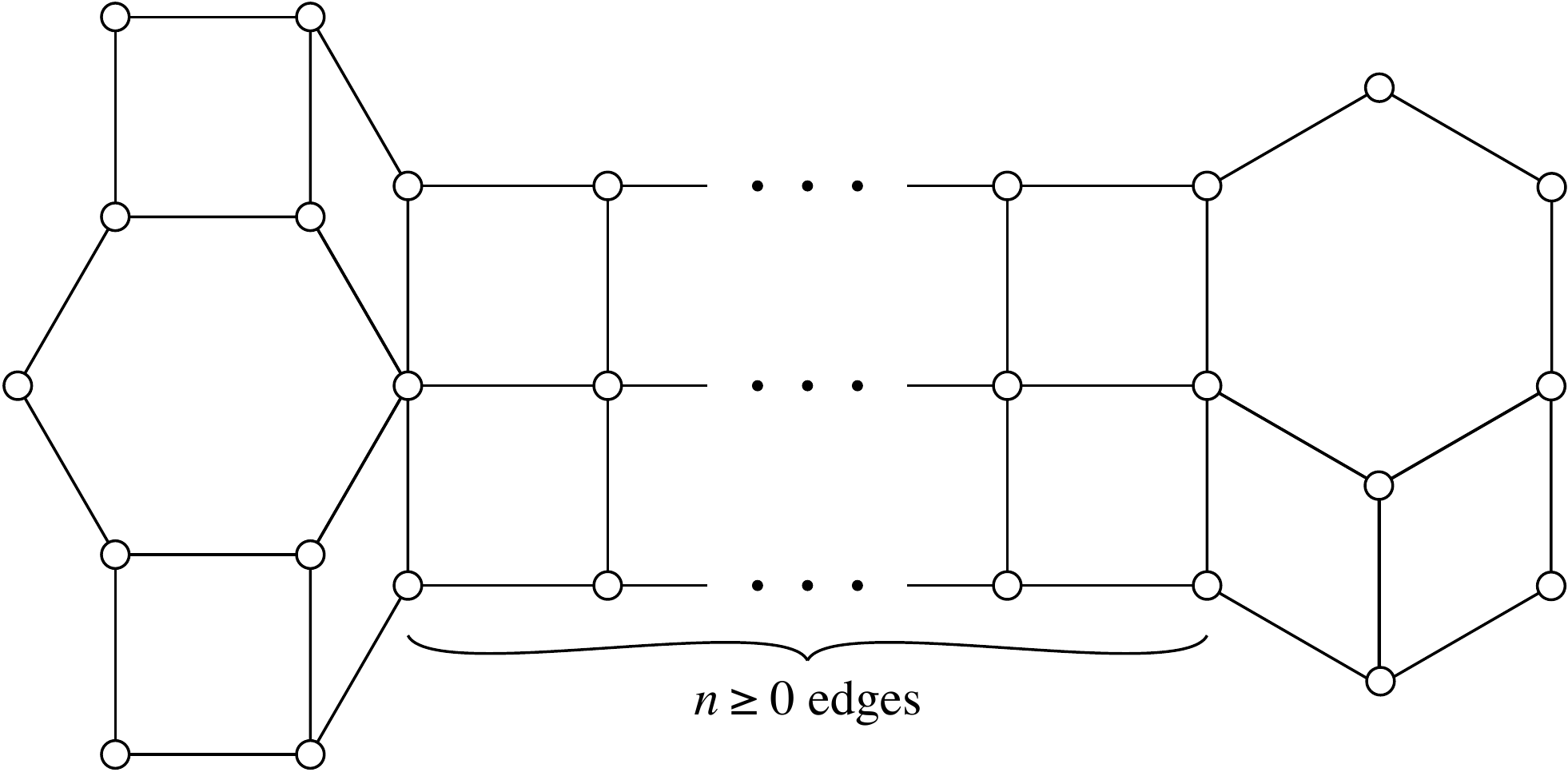}}
\caption{A critical isometric subgraph $H_n$ of the
Cartesian product of four trees}
\label{critical}
\end{figure}

\begin{figure}[t]
\scalebox{0.90}{\includegraphics{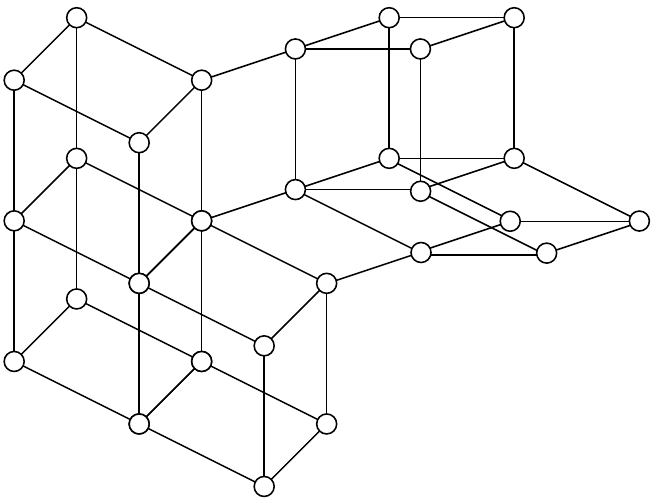}}
\caption{The median graph $G$ generated by the graph $H_1$ of Fig.~\ref{critical}}
\label{critical-median}
\end{figure}

\section{Proof of Theorem \ref{recognition}}

In  \cite{BaChEpp_square} we described how to recognize squaregraphs
in linear time by an algorithm based on breadth-first search. Now we
will show that partial double trees also can be recognized in linear
time by an algorithm based on lexicographic breadth-first search.

\emph{Breadth-First Search} (BFS) is one of the simplest algorithmic
ways to order the vertices of a connected graph $G$. In BFS, the vertices of the given graph  are ordered by
their distances from a given base point $b$. The algorithm begins by
placing vertex~$b$ into an otherwise empty queue. It then repeatedly
removes the first vertex $v$ from the queue, and
adds to the end of the queue any neighbor $w$ of $v$ that has not
already been added to the queue. The neighbor $w$ is referred to as
a son of $v$, and conversely $v$ is their father, written in short
as $f(w)=v$. When the queue becomes empty,
the parent-child relationship between the vertices of
$G$ forms a rooted spanning tree $T$ of $G$ with root $b$. Each
vertex is inserted and removed once from the queue, and each edge is
examined only when one of its endpoints is removed from the queue,
so the total time for this procedure is linear in the worst case.

\emph{Lexicographic Breadth-First Search} (LexBFS), proposed by
Rose, Tarjan, and Lueker \cite{RoTaLu,Cor-GMiCS-04}, is an
alternative procedure for traversing the vertices of a connected
graph $G$ in a more specific way: every LexBFS ordering is a BFS
ordering, but not necessarily vice-versa. In LexBFS, the queue of
vertices is replaced by a queue of sets of vertices, initially
containing two sets: the first set in the queue is $\{b\}$ and the
second contains all other vertices. In each iteration, the algorithm
removes an arbitrarily chosen vertex $v$ from the first set in this
queue, and removes the set itself if it becomes empty. When $v$ is
removed, the LexBFS algorithm partitions each
remaining set $S$ in the queue into two smaller sets, $S\cap N(v)$
and $S\setminus N(v)$; if both of these sets are nonempty, $S$ is
removed from the queue and these two sets replace it in the same
position, with $S\cap N(v)$ placed earlier in the queue than
$S\setminus N(v)$. With some care in the data structures used to
represent the queue of sets and to partition the sets, LexBFS can be
implemented so that, as with BFS, it takes linear time in the worst
case. If one numbers each vertex in the order it is removed, and defines the {\it label} $L(x)$ of a vertex $x$ to be
the list of its numbered neighbors (already removed from the queue), sorted from smallest to largest,
then LexBFS has the property that the sequence of vertex labels it
produces is lexicographically sorted so that if $L(y)$ is a proper
prefix of $L(x),$ then $x$ is labeled before $y.$ Note that this sorting
would become the usual (dictionary) lexicographic ordering if we inverted the
numbering by starting from $n$ rather than 1. The parent $f(x)$
of a vertex $x$ is the first vertex in $L(x)$; the vertex with empty
label has no parent and is the root $b$ of LexBFS tree of the graph
$G$.

Let $G=(V,E)$ be a connected bipartite graph with $n$ vertices. For
a total ordering $v_1,\ldots,v_n$ of vertices of $G$ and an index
$1\le i\le n,$ we denote by $G_i$ the subgraph of $G$ induced by
$v_1,\ldots,v_i.$ We will say that $G$ is {\it strongly
dismantlable} for the given ordering if the following two properties
are both satisfied for every vertex $v_i$ in $G$:
\begin{itemize}
\item[(1)] If $i>1,$ then $v_i$ has  either one or two neighbors in $G_i$.
\item[(2)] If $v_i$ has two neighbors in $G_i,$ then $v_i$ belongs to a unique 4-cycle of $G_i$.
\end{itemize}

\begin{lemma}
\label{contractible} If a finite connected bipartite graph $G$ is
strongly dismantlable, then the geometric realization $|G|$ is
contractible, and therefore, simply connected.
\end{lemma}

\begin{proof} If $v_i$ has only one neighbor $v'$ in $G_i,$
then mapping $v_i$ to $v'$ yields an elementary contraction from $|G_i|$ to
$|G_{i-1}|$. Otherwise $v_i$ belongs to a unique 4-cycle
$(v_i,v',u,v'',v_i)$ of $G_i.$  Mapping $v_i$ to $u$ yields a
retraction from $G_i$ to $G_{i-1}$ as well as an elementary
contraction of the square complex $|G_i|$ to the square complex
$|G_{i-1}|$.\footnote{More precisely, according to Whitehead's
definition of an elementary contraction in a simplicial
complex~\cite{Whitehead} (p.~247), this step can be represented as a
pair of elementary contractions in a triangulation of $|G|$
constructed by splitting each quadrilateral of $|G|$ arbitrarily
into two triangles.} By induction, we conclude that the square
complex $|G|$ is contractible.
\end{proof}

Observe that, in the LexBFS ordering (or more generally, in any BFS
ordering) of any bipartite graph, any 4-cycle $abcd$ must either be
ordered with two vertices of the same color first and last (as $a$,
$b$, $d$, and $c$) or with two vertices of one color first and the
other color last (as $a$, $c$, $b$, $d$). We say that the former
ordering, $a$, $b$, $d$, $c$, is a {\it proper ordering} of the
cycle and that the latter ordering, $a$, $c$, $b$, $d$, is an {\it
improper ordering}.

\begin{lemma}
\label{lem:sdpo} If $G$ is a strongly dismantlable graph for a
LexBFS ordering of~$G$, then all 4-cycles in $G$ are properly
ordered by that ordering.
\end{lemma}

\begin{proof}
Suppose to the contrary that a strongly dismantlable graph $G$
contains a 4-cycle $abcd$ ordered as $a$, $c$, $b$, $d$. For $G$ to
be strongly dismantlable, $b$ would have to be part of another
4-cycle $abce$, with $e$ earlier than $b$ in the ordering. In this
case $adce$ would be a 4-cycle containing $d$, as would $abcd$,
contradicting the uniqueness of the 4-cycle containing $d$ among
earlier-numbered vertices and therefore contradicting the assumption
that $G$ is strongly dismantlable. This contradiction shows that the
improperly ordered 4-cycle $abcd$ cannot exist.
\end{proof}

\begin{lemma}
\label{lem:strong-dis} A connected bipartite graph $G$ is strongly
dismantlable for a LexBFS ordering of~$G$  if and only if
\begin{itemize}
\item[(i)] all labels have size at most two,
\item[(ii)] for each $x$, if $L(x)=(y,z)$, then $|L(y)\cap L(z)|=1$, and
\item[(iii)] no two consecutive vertices in the ordering have equal labels of size two.
\end{itemize}
\end{lemma}

\begin{proof}
Condition (i) rephrases the condition (1) in the definition of
strong dismantlability.  Condition (ii) implies that each $v_i$ with
two neighbors in $G_i$ belongs to a unique properly ordered 4-cycle
in $G_i,$ and conditions (i) and (iii) (together with the fact that
LexBFS orders all vertices having the same label consecutively)
prevent $G$ from containing improperly ordered 4-cycles, so these
conditions together imply that each $v_i$ with two neighbors in
$G_i$ belongs to a unique 4-cycle in $G_i$. Thus, every graph
satisfying the conditions (i)-(iii) of the lemma is strongly
dismantlable.

Conversely, suppose that $G$ is strongly dismantlable. For any $v_i$
that has two earlier  labeled neighbors $y$ and $z$, the unique
4-cycle containing $v_i$ in $G_i$ must be properly ordered by
Lemma~\ref{lem:sdpo}, so its fourth vertex must belong to the labels
of both $y$ and $z,$ hence (ii) is satisfied. If condition (iii) is
not satisfied, then  two consecutive vertices in the ordering have
equal labels of size two, whence these two vertices together with
the vertices of their label form an improperly ordered 4-cycle,
violating Lemma~\ref{lem:sdpo}. Thus, every strongly dismantlable
graph meets all the conditions of the lemma.
\end{proof}

\begin{lemma}
\label{lem:label-intersection} Let $G$ be a median graph ordered by LexBFS. If $x$ is any vertex of $G,$
$y$ and $z$ both belong to $L(x)$, and
$y\ne z$, then $|L(y)\cap L(z)|=1$.
\end{lemma}

\begin{proof}
Both $y$ and $z$ must be located one step closer from the root $b$
of the LexBFS tree than the vertex $x$, because they have the same
color in the bipartition of $G$ and if their distances from $b$
differed by two, then $x$ would come between them in the LexBFS
ordering. Let $m$ be the  median of $b,y$, and $z$. Since $y$ and
$z$ are at distance two apart via a path through $x$, $m$ must be
adjacent to both $y$ and $z$. Thus, $m$ is one step closer to $b$
than $y$ and $z$, so it appears earlier than them in the LexBFS
ordering and $m\in L(y)\cap L(z)$. If $L(y)\cap L(z)$ contains a
second vertex $m',$  then the vertices $m',m,x,y,z$ induce a
$K_{2,3}$ subgraph, which is impossible in a median graph.
\end{proof}

The next lemma is closely related to Lemma~4
of~\cite{BaChEpp_square}, which states that the farthest vertices
from any given vertex in a cube-free median graph have degree at
most two, and the proof is essentially the same.

\begin{lemma}
\label{lem:pdd-2l} If $G$ is a cube-free median graph, and $G$ is
ordered by LexBFS, then the label of every vertex in $G$ has size at
most two.
\end{lemma}

\begin{proof}
Suppose for a contradiction that for some vertex $x$, $L(x)$
contains three vertices $u$, $v$, and $w$. By
Lemma~\ref{lem:label-intersection} the labels of each pair of these
three vertices have a nonempty intersection. If there is a vertex
that belongs to all three labels $L(u)$, $L(v)$, and $L(w)$, then
that vertex together with $u$, $v$, $w$, and $x$ induces a $K_{2,3}$
subgraph, which is impossible in a cube-free median graph. On the
other hand, if the three vertices in $L(u)\cap L(v)$, $L(u)\cap
L(w)$, and $L(v)\cap L(w)$ are distinct, then these three vertices
together with $u$, $v$, $w$, and $x$ induce a $Q_3-v$ subgraph,
again impossible in a cube-free median graph.
\end{proof}

\begin{lemma}
\label{lem:pdd-recog} A connected graph $G$ is a partial double tree
if and only if
\begin{itemize}
\item[(i)] $G$ is strongly dismantlable for the LexBFS ordering, and
\item[(ii)]for each vertex $v$, $\Link(x)$ is bipartite.
\end{itemize}
\end{lemma}

\begin{proof}
If a connected graph $G$ satisfies the conditions (i) and (ii), then
by Lemma~\ref{contractible} it is the underlying graph of a simply
connected rectangular complex in which by condition (ii) every
$\Link(x)$ is bipartite. Therefore it satisfies condition (iv) of
Theorem~\ref{ramified-polygon}, whence by condition (v) of the same
theorem, $G$ is the underlying graph of a partial double tree.

Conversely, suppose that $G$ is a partial double tree. By condition
(vi) of Theorem~\ref{partial-double-tree}, $G$ is a cube-free median
graph, in particular $G$ is connected and bipartite and satisfies
the conditions of Lemma \ref{lem:label-intersection} and Lemma
\ref{lem:pdd-2l}. By condition (v) of
Theorem~\ref{partial-double-tree}, it satisfies property (ii) of the
lemma. To show that $G$ is strongly dismanlable we verify that each
of the three conditions of Lemma~\ref{lem:strong-dis} is met.
Condition (i) of Lemma~\ref{lem:strong-dis} states that each label
has at most two vertices; this follows by Lemma~\ref{lem:pdd-2l}.
Condition (ii) of Lemma~\ref{lem:strong-dis} must also hold, by
Lemma~\ref{lem:label-intersection}. Finally, condition (iii) must
also hold, because if two vertices $x$ and $y$ both had equal labels
$(u,v)$ then the five vertices in $(L(u)\cap L(v))\cup\{u,v,x,y\}$
would induce a $K_{2,3}$ subgraph, impossible in a partial double
dendron.
\end{proof}

Now assume that $G=(V,E)$ is an arbitrary input graph and consider the following simple
recognition algorithm:

\medskip\noindent
{\bf Step (1):} Run the LexBFS algorithm for $G$ and check whether $G$ is connected
and bipartite. Test whether each vertex $v_i$ has one or two
previously labeled neighbors (i.e., $|L(v_i)|\le 2),$  and test whether each
two-vertex label $L(v_i)$ is different from $L(v_{i-1})$. For each two-vertex
label $L(v_i)=(y,z)$, test that $|L(y)\cap L(z)|=1$. Return the answer
``No" if $G$ fails to pass any one of these tests.

\medskip\noindent
{\bf Step (2):}  Make a list of 4-cycles of the form
$C_i=(v_i,y,z,w)$ where  $L(v_i)=(y,z)$ and $L(y)\cap L(z)=\{w\}$.
Associate with each edge of $G$ a list of the 4-cycles that contain
that edge, by initializing an empty list for each edge object and
then, for each 4-cycle, adding it to the lists of its four incident
edge objects.

\medskip\noindent
{\bf Step (3):} Using the lists returned by Step (2), for  each
vertex $v_i$ construct $\Link(v_i)$: the vertices of $\Link(v_i)$
are the edges in $G$ incident to $v_i$, and the edges incident to
each vertex of $\Link(v_i)$ are given by the 4-cycles incident to
the corresponding edge in $G$. Test whether each constructed graph
$\Link(v_i)$ is bipartite. If some link is not bipartite, then
return the answer ``No"; otherwise, return the answer ``Yes".

\medskip
First we show that if the algorithm returns the answer ``No", then
$G$ is not a partial double dendron. This is obviously true if $G$
is not connected or not bipartite. If the algorithm returns ``No''
because some vertex has a label that is too large, or because the
size of the intersection of two labels is not one, then $G$ fails
one of the conditions of Lemma~\ref{lem:strong-dis}, is not strongly
dismantlable, and hence also fails condition (i) of
Lemma~\ref{lem:pdd-recog}. Finally, if it returns ``No'' because
some vertex has a non-bipartite link, then it fails condition (ii)
of Lemma~\ref{lem:pdd-recog}. In each case, by
Lemma~\ref{lem:pdd-recog}, $G$ cannot be a partial double dendron.

Conversely, suppose that the algorithm returns the answer ``Yes''.
Then $G$ must be connected and bipartite.  The tests in Step (1)
check each condition of Lemma~\ref{lem:strong-dis}, so $G$ must pass
all conditions and must be strongly dismantlable, meeting condition
(i) of Lemma~\ref{lem:pdd-recog}. The test in Step (3) verifies that
each link is bipartite, so $G$ meets condition (ii) of
Lemma~\ref{lem:pdd-recog}. It follows from Lemma~\ref{lem:pdd-recog}
that $G$ is a partial double dendron. This concludes the proof of
Theorem~\ref{recognition}.

Note that we are using LexBFS only to ensure that vertices with
equal labels are ordered consecutively. A very similar algorithm
would work using BFS in place of LexBFS if we modified the test that
no two consecutive vertices have equal length-two labels to test
instead that no two vertices, consecutive or not, have equal
length-two labels. It would also be possible to use an algorithm
from \cite{Ep_arb} in Step~(2) to list all 4-cycles in the graph in
linear time, but this would add unnecessary complexity to the
algorithm as the 4-cycles are more easily obtained directly from the
labeling information gathered in Step~(1). Indeed, if  $G$ passes
the tests in Step (1), then the  list of 4-cycles constituted in
Step (2) must include all 4-cycles in $G$, because it includes all
properly ordered 4-cycles and because the test that rules out pairs
of consecutive equal two-vertex labels eliminates the possibility
that an improperly ordered 4-cycle might exist. Finally note that if
$G$ is a cube-free median graph ordered by LexBFS, $v_i$ is a vertex
of $G$ such that $L(v_i)=(y,z)$ and $L(y)\cap L(z)=\{x\},$ then $y$
is the father of $v_i$ in the LexBFS tree and $x$ is either the
father of both $y$ and $z$ or $x$ is the father of $z$
\cite{BaChEpp_square,Ch_CAT}. This {\it fellow traveller} property
was used in \cite{BaChEpp_square} together with BFS to recognize
squaregraphs, namely, to make a list of 4-cycles similar to that
constructed in Step (2) of our algorithm.

\section{Proof of Theorem \ref{cube-free}}

First we prove the result in the particular case when $\mathcal K$ is a square complex.
For finite square complexes, the equivalence (i)$\Longleftrightarrow$(ii) is essentially
a particular case of Theorems 3.13 and 3.16 of van de Vel \cite{VdV},
which assert that {\it a finite graph $G$ is  median if and only
if the cubical complex $|G|$ equipped with the $l_1$-metric
is a median space.} To conclude the proof in the finite case, it remains
to notice that the underlying graph of a median square complex $\mathcal K$ is cube-free.
Suppose by way of contradiction that $G({\mathcal K})$ contains a
3-cube $Q_3.$ Then only the 2-dimensional faces of $Q_3$ can
correspond to faces of ${\mathcal K}$.  Pick the center $u$ of one
such face and any two diametral points $v,w$ of the opposite face.
The triplet  $u,v,w$ contradicts the assumption that $\mathcal K$ is
median. This shows that the median graph $G({\mathcal K})$ is
cube-free.

To establish the equivalence (i)$\Longleftrightarrow$(ii) for infinite
square complexes, we will show how to adapt  van de Vel's result in
this more general framework. First suppose that $G({\mathcal K})$ is a
cube-free median graph. Pick three arbitrary points $x,y,z\in {\mathcal K}$ and three
shortest $l_1$ paths $\gamma(x,y),\gamma(y,z),$ and $\gamma(z,x)$  in $\mathcal K$ between
these points. Each of these paths intersects the interior of a finite number of squares of
$\mathcal K$ (otherwise, we can transform such a path to a path of the same length but
containing an infinite number of different vertices of $G({\mathcal K}),$ which is impossible).
Consider the convex hull $H$ in $G$ of the vertices of all squares of $\mathcal K$ whose interiors intersect
$\gamma (x,y)\cup \gamma (y,z)\cup \gamma (z,x)$ and of the edges of $\mathcal K$ contained in this union.
The set $H$ is finite as a convex hull of a finite set in a median graph \cite{VdV}. Moreover $H$ is a
median subgraph of $G.$
By Theorem 3.16 of \cite{VdV}, $|H|$ is a median subspace of $|G|={\mathcal K}.$ Therefore
the triplet $x,y,z$ admits a median point in $|H|,$ i.e., a point $m'$ of $|H|$ lying
simultaneously on shortest $l_1$ paths  $\gamma'(x,y),\gamma'(y,z),$ and $\gamma'(z,x)$ of $|H|$ between $x,y,$ and $z.$
From the choice of $\gamma(x,y),\gamma(y,z),$ and $\gamma(z,x)$ and the definition of $|H|$ we
conclude that   $\gamma'(x,y),\gamma'(y,z),$ and $\gamma'(z,x)$ are also shortest $l_1$ paths of
$\mathcal K$, whence $m'$ is also a median of $x,y,$ and $z$ in $\mathcal K.$ If $x,y,z$ admit another median
point $m'',$ then pick any shortest $l_1$ paths  $\gamma''(x,y),\gamma''(y,z),$ and $\gamma''(z,x)$ between $x,y,z,$ and
passing via $m''.$ Consider the finite median subgraph $H'$ of $G$ defined in the same way as $H$ but with respect
to  the six paths $\gamma'(x,y),\gamma'(y,z),\gamma'(z,x),\gamma''(x,y),\gamma''(y,z),$
and $\gamma''(z,x).$ Since $m',m''\in |H'|,$ $m'$ and $m''$ are distinct median points of $x,y,$ and $z$ in $|H'|,$ contrary
to the fact that according to Theorem 3.16 of \cite{VdV}, $|H'|$ is a median space. This shows that (i)$\Longrightarrow$(ii) for arbitrary square complexes.

To establish that  (ii)$\Longrightarrow$(i), let $\mathcal K$ be a square complex such that the metric space $({\mathcal K},d)$ is median.
First we show that $G({\mathcal K})$ is an isometric subspace of $({\mathcal K},d).$ Pick two vertices $x,y$ of $G({\mathcal K})$
and a shortest $l_1$ path $\gamma(x,y)$ in $\mathcal K$. We will transform $\gamma(x,y)$ into a path of  $G({\mathcal K})$
with the same length.  We can assume without loss of generality that $\gamma(x,y)$ does not pass via any vertex $z$ of $G({\mathcal K}),$ otherwise
we can apply our argument separately to the pairs $\{ x,z\}$ and $\{ z,y\}.$ If $\gamma(x,y)$ crosses two incident sides of a square $S$ of
$\mathcal K$, then replacing the subpath of $\gamma(x,y)$ contained in $S$
by the subpath (of the same length) between the intersection points of $\gamma(x,y)$ with the sides of $S$ and going instead via the common vertex of these sides,
we will obtain a shortest $l_1$ path $\gamma'(x,y)$ passing via a vertex of  $G({\mathcal K}),$ to which we can apply the argument above. Therefore we can suppose that if
$\gamma(x,y)$ intersects a square of $\mathcal K$, then it crosses its boundary in two opposite sides. As $x$ and $y$ are vertices, this is possible only
if $\gamma(x,y)$ is a path of the graph $G({\mathcal K}).$ Hence $G({\mathcal K})$ is an isometric subspace of  $({\mathcal K},d).$ Each square and each edge of
$\mathcal K$ are compact convex subsets of $({\mathcal K},d),$ therefore they are  gated sets \cite{VdV}. We assert that the gate of each vertex $v$ of
$G({\mathcal K})$ in this set is also a vertex of the underlying graph.
We will show this for squares $S$ of $\mathcal K$, the proof for edges is analogous. Let $v'$ be the gate of $v$ in $S$ and suppose
that $v'$ is an inner point of the edge $[a,b]$ of $S.$ Let $S'$ be the square of $\mathcal K$ sharing with $S$ the side $[a,b]$ and intersecting
some shortest $l_1$ path between $v$ and $v'.$
If the gate of $v$ in $S'$ belongs to a side incident to $[a,b],$ then necessarily the gate of $v$ in $S$ will be one of the vertices $a$ or $b$ and not $v'.$
So, assume that the gate  of $v$ in $S'$ belongs to the interior of the side of $S'$ opposite to $[a,b].$ Continuing the same reasoning with $S'$ instead of $S,$
we will construct a sequence squares,  such that the gates of $v$ in two consecutive squares of this sequence belong to the interior of the opposite sides
of the first square. This sequence is finite because each time the distance from $v$ to the respective gate decreases by at least by 1. This is obviously impossible,
because the last square of the sequence has $v$ as a vertex.  Hence, indeed the gate of $v$ in each  face or edge of
$\mathcal K$ is a vertex. To complete the proof, it remains to show that the median $m$ in $\mathcal K$ of any three vertices
$x,y,z$ of  $G(\mathcal K)$ is also a vertex of this graph.  Suppose by way of contradiction that $m$ is an inner point of some square
$S$ of $\mathcal K$ (the case when $z$ belongs to the interior of an edge is analogous).
The gates $x',y',z'$ of  $x,y,z$ in $S$ are all vertices of $S.$ Since $m$ belong to the interval between $x$ and $y,$ necessarily there exists a shortest
$l_1$ path between $x$ and $y$  traversing $x',m,$ and  $y'$ in this order. This is possible only if $x'$ and $y'$ are opposite corners of $S.$
Analogously, we deduce that $x',z'$ and $y',z$ are also opposite corners of $S.$ Since this is impossible, necessarily $m$ is a vertex of $G({\mathcal K}).$
As  $G({\mathcal K})$ is an isometric subspace of the median space $({\mathcal K},d),$ we conclude that  $G({\mathcal K})$ is a median graph. This establishes
the implication (ii)$\Longrightarrow$(i).

The equivalence (i)$\Longleftrightarrow$(iii) in the case of square complexes
is a particular case of a result of \cite{Ch_CAT,Ro} which establishes that the underlying graph of
a cubical complex $\mathcal K$ is median if and only if $\mathcal K$
equipped with the $l_2$-metric is $\CAT(0)$. On the other hand, Gromov \cite{Gr}
characterized cubical $\CAT(0)$ complexes in the following pretty
manner: \emph{A cubical complex ${\mathcal K}$  is $\CAT(0)$ if and only
if ${\mathcal K}$ is simply connected and satisfies the following
condition: whenever three $(k+2)$-cubes of ${\mathcal K}$ share a
common $k$-cube and pairwise share common $(k +1)$-cubes, they are
contained in a $(k+3)$-cube of ${\mathcal K}.$} Applying this
characterization to square complexes ${\mathcal K}$, the
combinatorial condition is equivalent to the assertion that
${\mathcal K}$ does not contain three squares which share
a common vertex and pairwise share common edges, i.e., that
$\mathcal K$ does not contain a vertex $x\in V({\mathcal K})$
containing a $C_3$ in $\Link(x).$ In view of previous equivalences,
this shows that (iv) is equivalent to the first three conditions of
Theorem \ref{cube-free}.

It remains to establish the result for rectangular complexes $\mathcal K$. Define a linear map between each rectangular face $R$ of $\mathcal K$ and
the respective face of the square complex $|G({\mathcal K})|.$ One can easily
show that the image of any $l_1$ path of $\mathcal K$ under the resulting piecewise linear map is an $l_1$ path of $|G({\mathcal K})|,$
and vice versa, the preimage of any $l_1$ path of  $|G({\mathcal K})|$ is an $l_1$ path of $\mathcal K$. Therefore, the equivalence (i)$\Longleftrightarrow$(ii)
follows from a similar result for square complexes established above. On the other hand, the equivalence (iii)$\Longleftrightarrow$(iv) is a direct
consequence of another result of Gromov \cite{BrHa,Gr} characterizing the polygonal complexes with $\CAT(0)$ $l_2$-metrics. According to this result,
a rectangular complex $\mathcal K$ with a finite number of isometry types of cells is $\CAT(0)$ if and only if  $\mathcal K$ does not contain a vertex $x\in V({\mathcal K})$
with a $C_3$ in $\Link(x).$ To establish the equivalence between the conditions (i),(ii) and the last two conditions (iii),(iv) notice that, in view
of the result of \cite{Ch_CAT,Ro} mentioned above, the underlying networks of rectangular complexes satisfying these conditions
are the same: they are median and cube-free. Finally, if any of the conditions of Theorem \ref{cube-free} holds, then  $\mathcal K$ equipped
with the intrinsic $l_1$-metric $d$ is median and, by Theorem 3.13(3)  of \cite{VdV} and its extension to arbitrary rectangular complexes, $\mathcal K$
coincides with the geometric realization $|N({\mathcal K})|$ of its network.  This concludes the proof of Theorem \ref{cube-free}.

\section{Proof of Theorem \ref{ramified-polygon}}
To establish the implication (i)$\Longrightarrow$(ii) suppose that
the rectangular complex $\mathcal K$ is isometrically embedded into
the Cartesian product $D=D_1\times D_2$ of two dendrons $D_1$ and
$D_2.$ We wish to show that the median $m$ of any three points $x,y,$ and $z$
of $\mathcal K$ taken in this double dendron must belong to $\mathcal K$. Note that
for any given finite set $F$ of points in $\mathcal K$ we can let the projections $F_1$ and
$F_2$ of $F$ onto $D_1$ and $D_2$ subdivide the line segments, so that $F_1$ and $F_2$ may
be taken as subsets of the vertex sets of $D_1$ and $D_2,$ respectively. In particular,
let $F$ be the set of all vertices of $\mathcal K$ together with $x, y,$ and $z.$ Therefore
the underlying network $N({\mathcal K})$ of the complex $\mathcal K$ is isometrically embedded
into the double tree network $N(D).$ Every shortest path in the network
$N(D)$ is also a shortest path in the double tree $G(D)$, and vice versa.
In particular, the isometric embedding between $N({\mathcal K})$ and $N(D)$ is also
one between the underlying graphs $G({\mathcal K})$ and $G(D).$ Therefore $G({\mathcal K})$
is a partial double tree and thus a median graph by Theorem \ref{cube-free}. Consequently,
the median $m$ of $x, y,$ and $z$ in $G({\mathcal K})$ also serves as the corresponding
median in the underlying double tree $G(D)$ and hence in $D.$


First note that (ii)$\Longleftrightarrow$(iii) holds by virtue of
Theorem \ref{cube-free}.
To show that (iii)$\Longrightarrow$(iv), suppose by way of
contradiction that for some vertex $x\in V({\mathcal K}),$ $\Link(x)$
contains an odd cycle of length $2k+1$ with $k\ge 1$. This means
that we can find $2k+1$ squares $R_1,\ldots,R_{2k+1}$ of
${\mathcal K},$ all sharing the vertex $x$ and constituting an odd
rectangular wheel.  Denote by $xx'_i$ the common edge of the
rectangles $R_i$ and $R_{i+1(mod ~2k+1)}.$ Now, let $B(x,\epsilon)$
be the closed neighborhood with radius $\epsilon>0$ centered at $x$
which is a partial double dendron. On each segment $I(x,x'_i)$ pick
a point $x_i$ located at the same distance $\delta<\epsilon$ from
$x$ such that for each $i$  the square $(x,x_i,y_i,x_{i+1(mod
~2k+1)})$ contained in the rectangle $R_i$ also belongs to the
closed ball $B(x,\epsilon).$ Since inside each rectangle $R_i$ the
distance is measured according to $l_1,$ we conclude that
$(x,x_i,y_i,x_{i+1(mod ~2k+1)})$ is a rectangle of the metric space
$({\mathcal K},d)$ (and thus of $(B(x,\epsilon),d)).$ As a
consequence, we derived an odd rectangular wheel in the partial
double dendron $B(x,\epsilon),$ which is impossible. This
establishes that (iii)$\Longrightarrow$(iv).

To show that (iv)$\Longrightarrow$(v), first notice that since
$\mathcal K$ is simply connected and the links of all vertices are
$C_3$-free, Theorem \ref{cube-free} implies that $G({\mathcal K})$ is a
cube-free median graph. Since each $\Link(x)$ is bipartite, the
graph $G({\mathcal K})$ does not contain odd  cogwheels, thus from Theorem
\ref{partial-double-tree} we infer that $G({\mathcal K})$ is a partial double
tree. The equality ${\mathcal K}=|G({\mathcal K})|$ follows from last assertion of
Theorem \ref{cube-free}.

Then, to show that (v)$\Longrightarrow$(i), notice that if
$G({\mathcal K})$ embeds isometrically into
the Cartesian product of two trees $T_1$ and $T_2,$ then all edges
of $G({\mathcal K})$ from the same cutset color-class have the same length in the
rectangular complex ${\mathcal K}.$ Now, transform each tree $T_i$
into a dendron $D_i$ having the same topology as $T_i$ and obtained by
replacing every edge of $T_i$ by a solid edge  of $D_i$ of length
equal to the length of the edges of ${\mathcal K}$ from the
respective color class. This leads to a natural mapping $\varphi$
from ${\mathcal K}$ to $D=D_1\times D_2.$ It can be checked in a
standard way that this mapping is an isometric embedding of
$\mathcal K$ into $D=D_1\times D_2$.

As for the implication (v)$\Longrightarrow$(vi), if $G=G({\mathcal K})$ is a partial
double tree, then $G$ can be regarded as a median subgraph of two trees $T_1$ and $T_2$ such
that the projection from $G$ to either tree is surjective (Theorem \ref{partial-double-tree}). Then
both trees are finite as ${\mathcal K}$ is finite. The pre-image $T$ of any leaf $t$ of $T_2$ under the
projection from $G$ to $T_2$ is a finite tree, the neighbors in $G$ of which form an isomorphic copy $U$ projected to the neighbor $u$ of $t$.
Then $G$ can be recovered as the convex expansion of $G-T$ along $U$ \cite{Mu1}. If $x$ is a vertex of $U$ having two neighbors  $y$ and $z$ in $U$
such that $x$ is the hub of some $k$-cogfan $H$ in  $G-T$ where $y$ and $z$ have degree 2 in $H$, then the convex expansion extends $H$
to a $(k+2)$-cogwheel in $G.$ Hence $k$ is necessarily even by Theorem \ref{partial-double-tree}. Evidently, ${\mathcal K}=|G|$ is the convex expansion of $|G-T|$ along the
finite dendron $|U|$ by the interval $[0,\lambda],$ where $\lambda$ is the weight of the edges between the subnetworks $G-T$ and $T.$ Note that any gated dendron in $\mathcal K$ is a finite dendron which can be regarded as the geometric realization of some finite tree in the network of some refinement of the complex $\mathcal K$. Therefore a straightforward induction establishes
(v)$\Longrightarrow$(vi).

The implication (vi)$\Longrightarrow$(vii) is trivial because gated expansions are particular instances of gated amalagamations.  Finally, to prove
the implication (vii)$\Longrightarrow$(iv), proceed by a trivial induction on the number of amalgamation steps.
Since amalgamation of median rectangular complexes ${\mathcal K}_1$ and ${\mathcal K}_2$ along gated sets evidently yields a median space,
condition (ii) of Theorem \ref{cube-free} holds, whence ${\mathcal K}$ is simply connected by Theorem \ref{cube-free}. The links of vertices
of $\mathcal K$ which do not belong to the dendron $A$ along which the last amalgamation was performed are as in the respective
constituent and hence bipartite according to the induction hypothesis. Link$_1(x)$ and Link$_2(x)$ of a vertex $x$ on $A$ in the two gated
constituents ${\mathcal K}_1$ and ${\mathcal K}_2$ are glued along the vertices representing the edges on $A$. Since all paths in Link$_1(x)$ and Link$_2(x)$
connecting any pair of the latter vertices have the same parity, no odd circle arises and therefore $\Link(x)$ is bipartite, too.
This completes the proof of
Theorem \ref{ramified-polygon}.

\section{Proof of Theorem \ref{automorphism}} By Frucht's Theorem~\cite{Fru} any group can be
represented as the automorphism group of some graph $F$ (which can be chosen to be
finite if the group is finite). We stipulate that the trivial group
is represented by a nontrivial graph, say, by an asymmetric graph
with six vertices or an asymmetric tree
with seven vertices (Fig.~\ref{asymmetric}). Since every automorphism maps components to
components, we may add an asymmetric tree with seven vertices in the
case that the graph did not yet have this tree as a component.

\begin{figure}[t]
\includegraphics[height=1in]{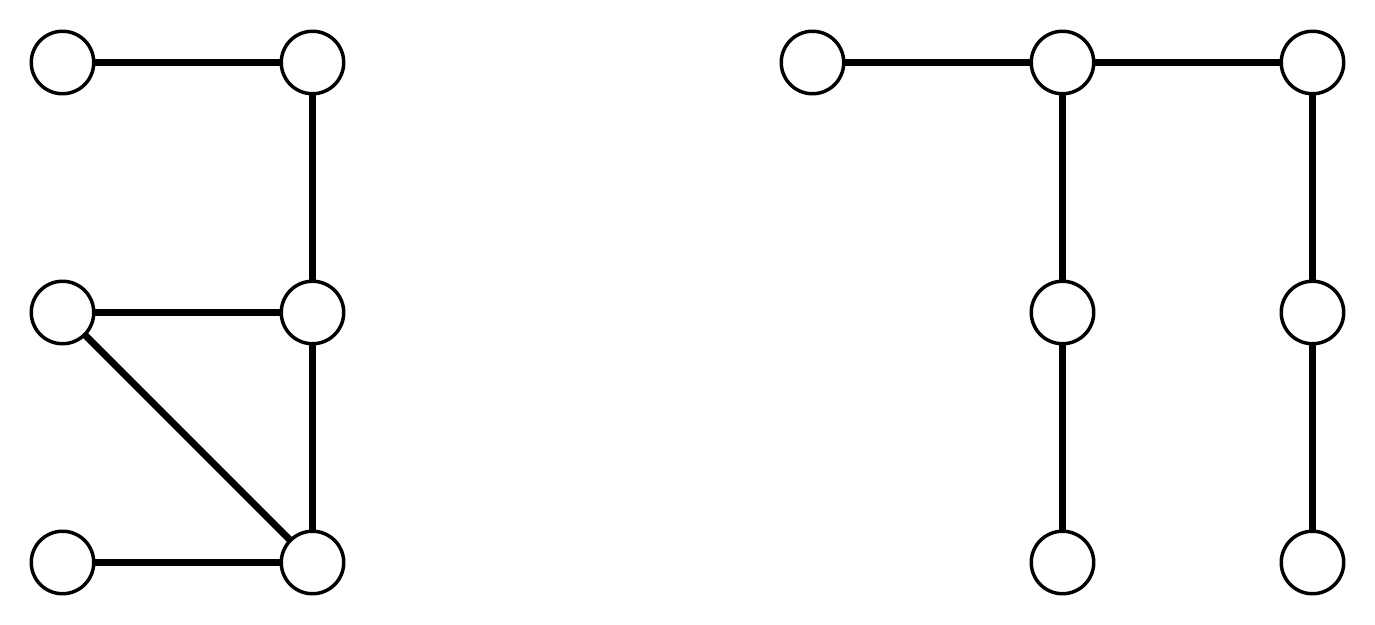}
\caption{A six-vertex asymmetric graph and a seven-vertex asymmetric tree}
\label{asymmetric}
\end{figure}

Next we claim that $F$ and its simplex graph $\kappa(F)$ have isomorphic
automorphism groups. Clearly every automorphism $f$ of $F$ maps
simplices to simplices of the same cardinality and trivially
preserves inclusion between simplices. Therefore $f$ lifts to the
simplex map $\kappa f$ which is an automorphism on $\kappa(F)$. Conversely, every
automorphism of $\kappa(F)$ that fixes the empty simplex is lifted from
some automorphism of $F.$ Since $F$ has a component with a convex
3-path,  all maximal hypercubes in $\kappa(F)$ intersect in a single
vertex, viz. the empty simplex node. This property guarantees that
every automorphism of $\kappa(F)$ fixes the latter vertex, because maximal
hypercubes are mapped onto maximal hypercubes by automorphisms of
$\kappa(F).$

Since $\kappa (F)$ is a non-singleton median graph, the iterated simplex graph $G=\kappa(\kappa(F))$ is 2-connected and has radius 2.
The automorphism groups of $\kappa(F)$ and $\kappa(\kappa(F))$ are isomorphic because the intersection of all maximal
hypercubes (4-cycles) in $\kappa(\kappa(F))$ is the empty simplex node of $\kappa(\kappa(F))$ (as $\kappa(F)$ includes
some 4-cycle). The incompatibility graph $\Inc(\kappa(\kappa(F)))$ of the convex splits of $\kappa(\kappa(F))$ is
isomorphic to the bipartite graph $\kappa(F),$ whence $\kappa(\kappa(F))$ is a partial double tree by Theorem \ref{partial-double-tree}.


It  remains to  prove the last assertion of Theorem \ref{automorphism}, that aut$(G)={}$aut$(|G|)$.
Clearly, every automorphism of $G$ extends to an automorphism of $|G|$; we must show that $|G|$ has no other automorphisms than the ones constructed in this way. To do so, we show that $G$ can be determined uniquely as a graph from the metric structure of $|G|$, without starting from any knowledge of the decomposition of $|G|$ into cells. However, this is straightforward: because of the construction of $G$ as a cube-free simplex graph, the 2-cells of $G$ are exactly the subsets that are isometric to $l_1$ unit squares and that have the additional property that at least two consecutive sides of the square consist of boundary points (points with a neighborhood homeomorphic to a half-plane whose boundary passes through the point). Because $G$ is a simplex graph of a graph without isolated vertices, every vertex or edge of $G$ is incident to a 2-cell in $|G|$, and the edges and vertices of $G$ are then exactly the sides and corners of these squares. Alternatively, one could also recognize the underlying graph $G$ here by the two observations that first the degree 2 vertices $u_i$ $(i\in I)$ of $G$ are exactly the points $p$ of    $|G|$ for which $|G|-\{ p\}$ is a median space and second every vertex $x$ of $G$ is    recognized via the requirement that $d(x,u_i)+d(x,u_j)-d(u_i,u_j)$ be an even integer for all $i,j\in I.$   This
concludes the proof of Theorem \ref{automorphism}.

\bibliographystyle{amsplain}
\bibliography{ramified}
\end{document}